\documentclass[oneside,english]{amsart}
\usepackage[T1]{fontenc}
\usepackage[utf8]{inputenc}
\usepackage{color}
\usepackage{amsthm}
\usepackage{amstext}
\usepackage{amssymb}
\usepackage{esint}
\usepackage{amsmath}
\usepackage[authoryear]{natbib}
\usepackage{booktabs}
\usepackage{graphicx}
\usepackage[colorlinks,backref]{hyperref}

\makeatletter
\theoremstyle{plain}
\newtheorem{thm}{\protect\theoremname}
  \theoremstyle{plain}
  \newtheorem{assumption}{\protect\assumptionname}
  \theoremstyle{remark}
  \newtheorem{rem}{\protect\remarkname}
  \theoremstyle{plain}
  \newtheorem{prop}[thm]{\protect\propositionname}

\theoremstyle{plain}
  \newtheorem{lem}[thm]{\protect\lemmaname}
  \newtheorem{example}[thm]{Example}

\usepackage{geometry}

\makeatother

\usepackage{babel}
  \providecommand{\assumptionname}{Assumption}
  \providecommand{\lemmaname}{Lemma}
  \providecommand{\remarkname}{Remark}
\providecommand{\theoremname}{Theorem}
\providecommand{\propositionname}{Proposition}

\def\R{\mathbb{R}}
\def \N{\mathbb{N}}

\def\C{\mathcal{C}}
\def\F{\mathcal{F}}%tribu
\def\G{\mathcal {G}}
\def\A{\mathcal {A}}
\def\J{\mathcal {J}}
\def\NN{\mathcal {N}}

\def\L{\mathcal{L}}%likelihood
\def\LL{\mathbb{L}}
\def\w{\omega}
\def\E{E}

\def\EE{\mathbf {E}} % esperance annealed

\def\1{\mathds{1}} % indicatrice

\def\t{\theta}
\newcommand{\ts}{{\theta^{\star}}}
\newcommand{\htet}{\bar \theta_t}

\def \eps{\varepsilon}

\def\to{\rightarrow}
\def\Xs{X_s}

\definecolor{lilas}{RGB}{182, 102, 210}
\newcommand{\bac}{\color{cyan}} %% arnaud
\newcommand{\argmax}{\mathop{\rm argmax}}

\def\dd{\mathrm{d}} %% pour le d des integrales
\def\to{\rightarrow}

\begin{document}

\title{Jump filtering and efficient drift estimation for L\'evy-driven SDE's}

\author{Arnaud Gloter, Dasha Loukianova and Hilmar Mai}

\address{Université d'Evry Val d'Essonne \\
91037 Évry Cedex \\
France}

\email{arnaud.gloter@univ-evry.fr}

\address{Université d'Evry Val d'Essonne \\
91037 Évry Cedex \\
France}

\email{dasha.loukianova@maths.univ-evry.fr}

\address{Centre de Recherche en Economie et Statistique\\
ENSAE-ParisTech\\
92245 Malakoff\\
France}

\email{hilmar.mai@ensae.fr}

\begin{abstract}
The problem of drift estimation for the
solution $X$ of a stochastic differential equation with L\'evy-type jumps
is considered under discrete high-frequency observations with a growing observation window.
An efficient and asymptotically normal estimator for the drift parameter is constructed under minimal conditions on the jump behavior and the sampling scheme. In the case of a bounded jump measure density these conditions reduce to
$n\Delta_n^{3-\eps}\to 0,$ where $n$ is the number of observations and $\Delta_n$ is the maximal sampling step. This result relaxes the condition $n\Delta_n^2 \to 0$ usually required for joint estimation of drift and diffusion coefficient for SDE's with jumps. The main challenge in this estimation problem stems from the  appearance of the unobserved continuous part $X^c$ in the likelihood function. 
In order to construct the drift estimator we recover this continuous 
part from discrete observations. More precisely, we estimate, in a nonparametric way,
stochastic integrals with respect to $X^c$.
Convergence results of independent interest are proved for these nonparametric estimators. 
Finally, we illustrate the behavior of our drift estimator  for a number of popular L\'evy--driven models from finance.
\end{abstract}


\maketitle

\section{Introduction}

The class of solutions of L\'evy-driven stochastic differential equations (SDE's) has recently attracted a lot of attention in the literature due to its many applications in various area such as  finance, physics and neuroscience. Indeed, it includes important examples taken from finance such as the well-known Barndorff-Nielsen-Shephard model, the Kou model and the Merton model (cf. \cite{BNS01}, \cite{kou} and \cite{merton1976}) as well as the stochastic Morris-Lecar neuron model (cf. for example \cite{dittlevsen2013}) from neuroscience to name just a few. Consequently, statistical inference for these models has recently become an active domain of research.  

In this work we aim at estimating the unknown drift parameter $\theta \in \Theta\subset \R^d$
based on discrete observations $X^{\t}_{t_{0}},\ldots,X^{\t}_{t_{n}}$ of the process $X^{\t}$  given by

\begin{equation}\label{eq:Xsde}
X_{t}^{\theta}=X_{0}^{\t}+\int_{0}^{t}b(\theta,X_{s}^{\theta})\; ds+\int_{0}^{t}\sigma(X_{s}^{\theta})\; dW_{s}+\int_{0}^{t}\gamma(X_{s-}^{\theta})\; dL_{s},\quad t\in\mathbb{R}_{+},
\end{equation}
where
   $W=(W_{t})_{t\geq0}$ is a one-dimensional Brownian motion and $L$  a pure jump Lévy
process
 with Lévy measure $\nu$.
 % satisfying $\int_{\R}|z|d\nu(z)<\infty.  $

We consider here the setting of high frequency observations with a growing time window, i.e. for the discrete sample $X^{\t}_{t_{0}},\ldots,X^{\t}_{t_{n}}$ with $0 \leq t_0 \leq \ldots \leq t_n$ we assume that the sampling step $\Delta_n := max \{ t_i - t_{i-1} : 1 \leq i \leq n\}$ tends to $0$ and $t_n \to \infty $ as $n \to \infty$. 
  It is well known that due to the presence of the diffusion part, one can only estimate the drift consistently   if $t_n\to\infty.$ A crucial point for applications in the high frequency setting is to impose minimal conditions on the sampling step size $\Delta_n$. This will be one of our main objectives in this paper. 

The topic of high frequency  estimation for discretely observed diffusions without jumps is well developed by now. See for example  \cite{yoshida1992}, \cite{kessler1997} and references therein for joint estimation of drift and diffusion coefficient.
 Less  results are known when a jump component is added to the process.  
 In the case of  high frequency  estimation for  diffusion with an additional jump component \cite{masuda2013} investigates Gaussian quasi-likelihood estimators  of a joint drift-diffusion-jump part parameter. \cite {shimizu2006}
define a contrast-type  estimation function, for joint estimation of  drift, diffusion and jump parts when the jumps are of compound Poisson type. \cite {shim} generalizes these results to include more general driving L\'evy processes.  
The LAN property for drift and diffusion parameters is studied in \cite{etEla} via Malliavin calculus techniques.
In all  these  papers joint estimation is considered under conditions on the sampling scheme and the L\'evy measure, which, in the case of a bounded jump measure density, is at best $n\Delta_n^2\to 0$.

It is important to note here that the principles of the estimation of  the drift, diffusion or jump law parameters are of completely different nature. The estimation of the volatility is feasible on a compact interval, whereas the estimation of the drift and the jump law requires a growing time window. Also due to the Poisson structure of the jump part the estimation of the jump parameter can be well separated from those of the drift and the diffusion part.
In this work we focus therefore on the estimation of the drift parameter only and construct a consistent, asymptotically normal and efficient estimator, under conditions on the jump behavior and the sampling scheme, which, in the case of  bounded jump measure density reduce to $n\Delta_n^{3-\eps}\to 0.$ 

A natural approach to estimate the unknown drift parameter would
be to use a maximum likelihood estimation, but the likelihood function based on the discrete sample is not tractable in this setting, since
it depends on the transition densities of $X$ which are not explicitly known. On the contrary, the continuous-time likelihood function is explicit.
 Our aim is to approximate this function from discrete data and hence define some contrast function. The main difficulty is that the continuous-time likelihood  involves the continuous  part $X^c$ of $X$ that is unobservable under discrete sampling.
 Intuitively, this tells us that the continuous part $X^c$ has to be recovered, hence the jumps of $X$ have to be removed 
 in order to obtain an approximation of the continuous likelihood function.

 The  question of estimation of the continuous part of an Itô-semimartingale appears naturally in many statistical inference questions (cf. for example \cite{Mancini2011} and \cite{bibinger2015}) and constitutes in itself an interesting nonparametric problem. 
 In this article we  study the question of estimation of  stochastic integrals with respect to the continuous part of $X$ from a discrete sample of $X$. Propositions \ref{prop:jumpfiltering} and \ref{prop:jumpfiltering_infinit}  give explicit rates of convergence for our estimators of these quantities. Besides being of independent interest these results constitute the main tool for  the asymptotic analysis of our drift estimators.

The technique
  we use in order to recover stochastic integrals with respect to the continuous part of $X$  consists in comparing  the  increments of $X$  with a threshold $v_n$, suggested by the typical behavior of a diffusion path. This approach will be called
  jump filtering in the sequel.   
  Similar ideas of thresholding  were also used in \cite{shimizu2006}, \citet{Mancini2011}, \citet{mai2014} and \cite{bibinger2015}. In this article we have paid particular attention to the study of the joint law of the biggest jump and of the total contribution of the other jumps in each sampling interval (Lemma \ref{L:law_sum_cond_big}),  which permits us to improve existing conditions on the sampling scheme in the drift estimation problem.

The drift estimator is then constructed by applying a jump filter to the discretized likelihood function and maximizing the resulting criterion function to obtain what will be called the filtered MLE (FMLE). To study the properties of the FMLE
we first focus on the MLE obtained from continuous observations and show that 
this  MLE is asymptotically normal (Theorem \ref{thm:cont_clt}) with explicit asymptotic variance. We then  prove the LAN property which gives by Hàjek-Le Cam's convolution theorem that the continuous MLE is efficient (Theorem \ref{thm:LAN}). 
We show in the next step that the FMLE attains asymptotically the same distribution as the MLE based on continuous observations, which proves the efficiency of the FMLE (Theorems \ref{thm:anfinact}, \ref{thm:aninfinact}). The last step is mainly based on our results for the jump filter (Propositions \ref{prop:jumpfiltering} and \ref{prop:jumpfiltering_infinit}).

 The consistency of the FMLE is obtained without further assumptions on the sampling scheme.  The asymptotic normality necessitates some additional conditions on the rate at which $\Delta_n$ goes to $0$ that depend on the behavior of the L\'evy measure $\nu$
near zero. In the case where $\nu$ has a bounded Lebesgue density these conditions reduce to $n\Delta_n^{3-\eps}\to 0$ for some $\eps>0.$  We believe that this condition is unavoidable, because it is already necessary  in the Euler discretization scheme of the stochastic integral with respect to $X^c$ (Lemma \ref{lem:euler}). It is in accordance with the condition $n\Delta_n^3\to 0$ of \cite {Flo-Zim} in the case of drift estimation for continuous diffusions, hence our result can be seen as a generalization of \cite {Flo-Zim} to the presence of jumps.

In the literature on joint estimation of drift and diffusion parameters for models with diffusion and jump part the condition $n\Delta_n^{2}\to 0$ is usually required (cf. \cite{masuda2013},
\cite{shimizu2006} and \cite {shim}). The same condition on the sampling scheme appears for joint estimation in the case of continuous diffusions in \cite{yoshida1992}. Hence, our work shows that by focusing on drift estimation the condition $n\Delta_n^2\to 0$  can  be relaxed  in the presence of jumps as well.
    
%and  and local asymptotic normality of the statistical experiment we conclude by Hàjek-Le Cam's convolution theorem that the resulting asymptotic variance is efficient. This asymptotic variance then serves us as a benchmark for the case of discrete observations.  
%% This approach leads also to a fundamental understanding of what are the intrinsic difficulties of the problem and how an efficient estimator can be achieved. 

%The estimation of $\theta$ is particularly challenging in the presents of jumps, since two fundamentally different sources of noise have to be handled. The first source is the diffusion component in \eqref{eq:Xsde} driven by $W$ that represents a continuous, additive and centered noise. The second source is the jump component $\int \gamma \; dL$ that gives a purely discontinuous noise that asks for a completely different treatment. This dichotomy appears also in 

%
%Our proof of efficiency starts from the likelihood function based on continuous observations. 

 As will be seen in Section \ref{sec:applications} many popular models lead to explicit estimators, which do not require the knowledge of the diffusion coefficient and that perform well in numerical examples.

The structure of the paper is as follows. In Section \ref{sec:assumptions} the problem setting and the main assumptions of this work are introduced. Section \ref{sec:main} contains the construction of the drift estimator from discrete observations together with the main results. In Section \ref{sec:jumpfilter} we discuss the approximation of the continuous martingale part and prove the convergence of the jump filter. Section \ref{sec:applications} is devoted to applications to popular parametric jump diffusion models and some numerical examples. Finally, in Section \ref{sec:mainproofs} and \ref{sec:proof_filtering} we prove the main results and the convergence of the jump filter respectively, and Section \ref{sec:auxlemmas} contains some auxiliary results that are frequently used in the sequel.

%%%%%%%%%%%%%%%%%%%%%%%%%%%%%%%%%%%%%%%
\section{ Model, assumptions and ergodicity } \label{sec:assumptions}
%%%%%%%%%%%%%%%%%%%%%%%%%%%%%%%%%%%%%%%%%
% existence, ergodicity  %%
 Let  $\Theta$ be a compact subset of $\R^d$ and $X^{\t}$ a solution to \eqref{eq:Xsde} which can be rewritten as
\begin{equation*}
X_{t}^{\theta}=X_{0}^{\t}+\int_{0}^{t}b(\theta,X_{s}^{\theta})\; ds+\int_{0}^{t}\sigma(X_{s}^{\theta})\; dW_{s}+\int_{0}^{t}\int_{\R}\gamma(X_{s-}^{\theta})z\mu(ds,dz),\quad t\in\mathbb{R}_{+}, 
\end{equation*}
where
   $W=(W_{t})_{t\geq0}$ is a one-dimensional Brownian motion and $\mu$ is the Poisson random measure on $[0,\infty)\times \R$ associated with the jumps of the L\'evy process $L=(L_t)_{t\geq 0}$ 
   %$L_t=\int_0^t\int_{\R}z\mu(ds,dz)$  a Lévy
 with Lévy-Khintchine triplet $(0,0,\nu)$ such that $\int_{\R}|z|d\nu(z)<\infty.$ The initial condition $ X^{\t}_0
$, $W$ and $L$ are independent. We assume without
loss of generality that $0\in\Theta$ and $b(0,\cdot)\equiv0.$ 
\subsection{Assumptions}
We suppose that the functions $b:\Theta\times\R\to\R$, $\sigma:\R\to\R$ and $\gamma:\R\to\R$ satisfy the following assumptions:

\begin{assumption}[]\label{ass:existence}
The functions $\sigma(x), \gamma(x)$ and
for all $\t\in\Theta,$ $b(\t,x)$ are globally Lipschitz.
 Moreover, the Lipschitz constant of $b$ is uniformly bounded on $\Theta$.

\end{assumption}
 Under Assumption \ref{ass:existence} equation \eqref{eq:Xsde} admits a unique non-explosive càdlàg adapted solution possessing the strong Markov property, cf. \cite{applebaum}(Theorems 6.2.9. and 6.4.6). 
\begin{assumption}
\label{ass:recurrence}
 For all $\t\in\Theta$ there exists a constant $t>0,$ such that  $X^{\t}_{t} $ admits a density $p^{\t}_{t}(x,y)$ with respect to the Lebesgue measure on $\R$; bounded in $y\in\R$ and in $x\in K$ for every compact $K\subset \R$. Moreover, for every $x\in\R,$ and every open ball $U\in\R$ there exists a point $z=z(x,U)\in \mbox{supp}(\nu)$ such that $\gamma(x)z\in U.$
\end{assumption}
The last Assumption was used in \cite{Mas} to prove the irreducibility of the process $X^{\t}$. See also \cite{MasudaE} for other sets of conditions, sufficient for irreducibility.

\begin{assumption}[Ergodicity]
\label{ass:ergodic}
\begin{description}
\item [(i)] For all $q>0$,  $\int_{|z|>1}|z|^{q}\nu(dz)<\infty$.
\item [(ii)]For all $\t\in\Theta$ there exists a constant $C>0$ such that $xb(\t,x)\leq -C |x|^2,$ if $|x|\to\infty$.
\item [(iii)] $|\gamma(x)|/|x|\to 0$ as $|x|\to\infty.$ 
\item [(iv)] $|\sigma(x)|/|x|\to 0$ as $|x|\to\infty.$
\item[(v)] $\forall \t\in\Theta,$ $\forall q>0$ we have $\E|X_0^{\t}|^q<\infty.$
\end{description}
\end{assumption}
Assumption \ref{ass:recurrence} ensures together with Assumption \ref {ass:ergodic} the existence of unique invariant distribution $\pi^{\t}$, as well as the ergodicity of the process $X^{\t}$, as  stated in Lemma \ref{lem:ergodic} below. %Assumption \ref{ass:recurrence} corresponds to Assumption $2$ in \cite{Mas},
%The point (iii) of the Assumption \ref{ass:ergodic} permit to the jump coefficient to be sub-linear instead of bounded in \cite{Mas} (compare with formulae $11$, Lemma $2.4.$ in \cite{Mas}). 
%More precisely, the ergodic theorem holds for any $\pi $-integrable function a.s., and in mean for functions with polynomial growths.
%%%%%%%%%%%%%%%%%%%%%%%%%%%%%%%%%%%%%%%%%%%%%%%%%%%%%%%%%%%%%%%%%%%%%%%

%%%%%%%%%%%%%%%%%%%%%%%%%%%%%
%Levy measure
%%%%%%%%%%%%%%%%%%%%%
%%
\begin{assumption}[Jumps]\label{ass:jumps}
\begin{description}
	\item[(i)] The jump coefficient $\gamma$ is bounded from below, i.e.
	$ \inf_{x\in\R}|\gamma(x)|:=\gamma_{min}>0$
	(wlog we suppose  $\gamma_{min}\geq 1$).
	
\end{description}

%If the intensity of the Lévy process is infinite $\nu(\mathbb{R})=\infty$, we assume furthermore,
\begin{description}
	\item[(ii)] the L\'evy measure $\nu$ satisfies
	$\int_{0<|z|\leq 1}|z|\nu(dz)<\infty$,
	\item[(iii)] the L\'evy measure $\nu$  is absolutely continuous with respect to the Lebesgue measure,
	\item[(iv)] the jump coefficient $\gamma$ is upper bounded, i.e.
	$ \sup_{x\in\R}|\gamma(x)|:=\gamma_{max}<\infty$.
\end{description}	

 \end{assumption}
Note that the integrability condition given by the Assumption \ref{ass:jumps} (ii) is automatically satisfied  in the finite activity case $\nu(\R)<\infty$. This condition insures that the trajectories of the driving L\'evy process $L$  are a.s. of finite variation and hence the integral with respect to  $L$ in \eqref{eq:Xsde} can be defined as a deterministic Lebesgue-Stieltjes integral.  The third and  the fourth point of the Assumption \ref{ass:jumps} are technical and need in the infinite activity case.

%%%%%%%%%%%%%%%%%%%%%%%%%%%%%%%%%%%%%%%%%%%%%%%%%%%%%%

%Assumptions for likelihood:
%%%%%%%%%%%%%%%%%
%%%%%%%%%%%%%%%%%%%
The following assumption insures the existence of the likelihood function.
\begin{assumption}[Non-degeneracy]\label{ass:girsanov}
There exists some $\alpha >0,$ such that $\sigma^2(x)\geq \alpha$ for all $x\in\R.$

\end{assumption}

 \begin{assumption}[Identifiability]\label{ass:ident}
For all $\t\neq \t',$ $(\t,\t')\in\Theta^2,$
\begin{equation*}
\int_{\R}\frac{(b(\t,x)-b(\t',x))^2}{\sigma^2(x)}d\pi^{\t}(x)>0
\end{equation*}
 \end {assumption}
 We can see (cf. Proposition \ref {prop:ident}) that this last assumption is equivalent to 
 \begin{equation}\label{eq:idb}
 \forall \t\neq \t',\quad(\t,\t')\in\Theta^2,\quad b(\t,.)\neq b(\t',.).
 \end{equation}
 For $f:\Theta\to\R$ denote by
  $\nabla_{\t}f:\Theta\to\R^d$ the gradient column vector  and by $ \partial^2_{\t}f:=\left(\partial_{\theta_{i},\theta_{j}}^2f\right)_{1\leq i,j\leq d}$ the Hessian matrix of $f$. 
  We define $|\t|$ as the Euclidian norm of $\t\in\R^d$, and $|\partial_{\theta}^2f|:=\sqrt{\sum_{i,j=1}^n|\partial_{\theta_{i},\theta_{j}}^2f|^2}$ as
the Euclidian norm of the Hessian matrix of $f$.
%%%
The following assumption is used to insure the uniform in $\theta$ convergence needed in the proofs of  consistency and asymptotic normality:
 \begin{assumption}[Hölder-continuity of drift]\label{ass:hoelder}
 \begin{description}
 \item[(i)]
  For all $x\in\R$,  $b(.,x)$  is Hölder-continuous with respect to $\t\in\Theta$:
 \[\forall \theta, \theta',\quad |b(\theta,x)-b(\theta',x)|\leq K(x)|\t-\t'|^{\kappa},\]
 where $0<\kappa \leq 1$ and $K:\R\to\R_+$ is at most of polynomial growth.
 \item[(ii)]
   For all $x\in\R,$ $b(.,x)$ is twice continuously differentiable with respect  to $\theta$ and $\nabla b(.,x)$ and $\partial^2b(.,x)$ are Hölder-continuous with respect to $\t\in\Theta:$ 
   %for all $(i,j)\in\{1,\ldots,d\}^2$ :%\todo{We need this multidimensional. See also the proofs in Section 6.}
 %%
  \begin{align*}
& \forall \theta, \theta',\quad |\nabla b(\theta,x)-\nabla b(\theta',x)|\leq K_1(x)|\t-\t'|^{\kappa_1} \\
& \forall \theta, \theta',\quad |\partial^2_{\theta}b(\theta,x)-\partial^2_{\theta} b(\theta',x)|\leq K_2(x)|\t-\t'|^{\kappa_2} \\
\end{align*}
where  $0<\kappa_1, \kappa_2\leq 1$ and $ K_1, K_2:\R\to\R_+$ are at most of polynomial growth. 
\end{description}
\end{assumption}
We also need the following technical assumption:
\begin{assumption}\label{ass:subpolynom}
The functions $b,\sigma, \nabla_{\t} b,\partial_{\t}^2 b$ are twice continuously differentiable with respect to $x$. The functions
$  \sigma',\ \sigma''$ as well as the functions
\[x\mapsto \sup_{\t\in\Theta}|\frac{\partial^{i+j}b(\t,x)}{\partial^{i}x\partial^j\t}|\]
%\sup_{\t\in\Theta}|b(\t,x)|,\  \sup_{\t\in\Theta}|b'(\t,x)|,\  \sup_{\t\in\Theta}|b''(\t,x)|,\ 
are sub-polynomial for all $0\leq i\leq 2$ and $0\leq j\leq 2.$
\end{assumption}
Define
the asymptotic Fisher information 
by
\begin{equation}\label{eq:fisher}
I(\t)=\left (\int_{\mathbb{R}}\frac{\partial_{\t_i}b(\theta,x)\partial_{\t_j}b(\theta,x)}{\sigma^2(x)}\pi^{\t}(dx)\right)_{1\leq i,j\leq d}.
\end{equation}

\begin{assumption}\label{ass:fischer} For all $\t\in\Theta,$
$I(\t)$ is non-degenerated.
\end{assumption}

\subsection {Ergodic properties of solutions}
In all our statistical analysis an important role is played by ergodic properties of solutions of equation \eqref{eq:Xsde}. The following lemma is a generalization of a result of \cite{Mas}. It states conditions for the existence of an invariant measure $\pi^{\t}$ such that an ergodic theorem holds and moments of all order exist. 
A proof is given in Section \ref{sec:auxlemmas}.
\begin{lem} \label{lem:ergodic} Under assumptions \eqref{ass:existence} to \eqref{ass:jumps},  for all $\t\in\Theta,$ $X^{\theta}$ admits a unique invariant distribution $\pi^{\theta}$ and the ergodic theorem holds:
\begin{enumerate}
\item
 for every measurable function $g:\R\to\R$ satisfying $\pi^{\t}(g)<\infty,$ we have $a.s.$
\[\lim_{t\to\infty}\frac 1t\int_0^tg(X_s^{\t})ds=\pi^{\t}(g).\] 
\item For all $q>0,$ $\pi^{\t}(|x|^q)<\infty$.
\item For all $q>0,$ $\sup_{t\in\R}E[|X_t^{\t}|^q]<\infty$ and $\sup_{t\in\R}E[|X_{t_-}^{\t}|^q]<\infty$. 
\item 
Moreover, 
 \[\lim_{t\to\infty}\frac 1t\int_0^tE[|X_s^{\t}|^q] ds=\pi^{\t}(|x|^q).\]
 \end{enumerate}
 
\end{lem}
 % \begin{rem}\label{rem:integrab}Note that the polynomial growth of $K, K_i$ given by Assumption \ref{ass:hoelder} together with Assumption \ref{ass:girsanov} and the fact that $\pi^{\t}$ possess all moments
%(cf. Lemma \ref{lem:ergodic}),  provide that $\frac{K}{\sigma},\frac{K_i}{\sigma}\in\LL^p(\pi^{\t})$ $ i=1,2$, for all $p>0.$ In particular,
%for all $(\t,\t')\in\Theta^2,$ the functions $ b(\t,.),\ {\nabla b(\t,.)}$, ${\partial^2 b(\t,.)}$ 
%are at most of polynomial growth and $\frac{b(\t,.)}{\sigma(.)},\frac{\nabla b(\t,.)}{\sigma (.)},\frac{\partial^2 b(\t,.)}{\sigma (.)}$ belong to $\LL^p(\pi^{\t'})$ for all $p>0.$
%\end{rem}

%Denote by $\pi$ the stationary distribution of $X$.
%%
%%

%%%%%%%%%%%%%%%%%%%%%%%%%%%%%%%%%%%%%%%%%
\section{ Construction of the estimator and main results}\label{sec:main}
We define a discrete approximation to the continuous
time likelihood function by employing
a jump filtering technique and hence obtain an approximate maximum
likelihood estimator.
% $\hat{\theta}_{n}$ for $\ts$. 
We prove that this drift estimator
attains asymptotically the same performance as the maximum likelihood
estimator based on continuous observations under suitable assumptions on the jump behavior of the   driving L\'evy process $L$.
% In doing so we generalize the results developed in \citet{mai2014} to the context of general jump diffusion processes.
% Car cité dans l'intro ou cela est plus pertinent

\subsection{Construction of the estimator}
Let $X^{\t}$ be given by \eqref{eq:Xsde}. We denote by $P^{\t}$ the law of $X^{\t}$ on the  Skorokhod
space $D[0,\infty)$ of real-valued c\`ad l\`ag functions, and $P^{\t}_t$ its restriction on $D[0,t)$.
From now on we denote the true parameter value by $\ts$, an interior point of the parameter space $\Theta$ that we want to estimate. We shorten $X$ for $X^{\ts}$ and $P,E,\pi$ for respectively $P^{\ts}, E^{\ts},\pi^{\ts}.$
Suppose  that we observe a finite sample  
\begin{equation}
X_{t_{0}},\ldots,X_{t_{n}};\quad0=t_{0}\leq t_{1}\leq\ldots\leq t_{n}\label{eq:disc_obser}.
\end{equation} 
Every observation
time point depends also on $n$, but to simplify our notation we suppress this index.
We will be working in a high-frequency setting, i.e. 
\begin{equation*}
\Delta_n:=\sup_{i=0,\dots,n-1} (t_{i+1}-t_i) \xrightarrow{n \to \infty} 0.
\end{equation*}
We assume  $\lim_{n\to\infty} t_n= \infty$ %and $\quad \lim_{n\to\infty}\Delta_n=0$ such that
and  $n\Delta_n=O(t_n)$ as $n \to \infty$.
Under Assumption \ref{ass:girsanov}, $P^{\t}_t$ and $P^0_t$
%restricted to $\mathcal{F}_{t}$ 
are
mutually locally absolutely continuous for any $\theta \in \Theta$ (cf. for example \cite{JS}) and the likelihood function
is given by \begin{equation}
\L_{t}(\theta,X)= 
\frac{dP_{t}^{\theta}}{dP_{t}^{0}}(X)=\exp\left(\int_{0}^{t}\sigma(X_{s})^{-2}b(\theta,X_{s})\; dX_{s}^{c} -\frac{1}{2}\int_{0}^{t}\sigma(X_{s})^{-2}b(\theta,X_{s})^{2}\; ds\right).\label{eq:likelihood}
\end{equation}
 We define the log-likelihood function as 
 \begin{equation}\label{eq:loglikelihood} 
 \ell_t(\t):=\ln {\mathcal L}_{t}(\theta,X).
 \end{equation}
 
%and $$\htet(\t,X^{\t})\in\argmax_{\t\in\Theta}\ell_t(\t, X^{\t})$$
%the maximum likelihood estimator of $\t$.
The crucial point here is the appearance of $X^{c}$ in (\ref{eq:likelihood}),
since when $X$ is observed discretely, its continuous  part  remains
unknown. To handle this problem we use  a jump filter as  described below. 

For $ g:[0,t_n]\to \R,$ set $\Delta_i^n g=g_{t_{i}}-g_{t_{i-1}},\ i=1,\ldots n.$ In particular,
 $\Delta_{i}^nX=X_{t_{i}}-X_{t_{i-1}},$ $\Delta_{i}^nX^c=X^c_{t_{i}}-X^c_{t_{i-1}}$ and $\Delta_i^n Id=t_i-t_{i-1}.$
Let $\eps\in(0,1/2)$ and denote %More specifically we make the following choice for $v_n$. The
%cut-off sequence satisfies 
\begin{equation}\label{eq:v_ni}
v_n=\Delta_{n}^{1/2-\eps},\ n \geq 1.
\end{equation}
Define a discrete, jump-filtered approximation $\ell_{t_{n}}^{n}$  of the $\log$-likelihood function as follows.
\begin{equation}
\ell_{t_{n}}^{n}(\theta)=\sum_{i=1}^{n}\sigma(X_{t_{i-1}})^{-2}b(\theta,X_{t_{i-1}})\Delta_{i}^nX\mathbf{1}_{|\Delta_{i}^nX|\leq v_n}-\frac{1}{2}\sum_{i=1}^{n}\sigma(X_{t_{i-1}})^{-2}b(\theta,X_{t_{i-1}})^{2}\Delta_{i}^nId.\label{eq:like_filtered}
\end{equation}
The cut-off sequence $(v_n)$ is
 chosen in order to asymptotically filter the increments of $X$ containing jumps. The increments of the continuous martingale
part are typically of the order $\Delta_{n}^{1/2}$, which leads  to the  definition \eqref{eq:v_ni}.
The challenge now is to find suitable conditions on $\Delta_n$, $\epsilon$ and $\nu$ to make the likelihood  \eqref{eq:loglikelihood} well approximated by its discretized and jump filtered counterpart \eqref{eq:like_filtered} even in the case of infinite activity.
Of course we can choose $\eps$ arbitrarily small, which is a choice we have in mind.
Finally,
we define an estimator $\hat{\theta}_{n}$ of $\ts$  as 
\begin{equation}\label{eq:FMLE}
\hat{\theta}_{n}\in\operatorname*{argmax}_{\theta \in \Theta} \ell_{t_{n}}^{n}(\theta)
\end{equation}
and in the sequel we call it the filtered MLE (FMLE).

\subsection{Main results}
%We consider  for our convergence analysis of $\hat \theta_n$ two different settings (finite or infinite activity) for the jump component $L$, since the jump activity of $L$ has important implications
%when studying the discretization error.
 %implications on the discretization scheme when discretizing the criterion function. 
%The jump component $L$ is of compound Poisson type in the case of finite activity.
The following theorem gives a general consistency result for the FMLE $\hat{\theta}_{n}$ that holds for finite and infinite activity without further assumptions on $n$, $\Delta_n$ and $v_n$.
%%%%%%%%%
\begin{thm}[Consistency]
\label{thm:tnconsitency}
Suppose that Assumptions \ref{ass:existence} to \ref{ass:subpolynom} hold, 
then the FMLE $\hat{\t}_n$ is consistent
 in probability:
\[\hat\t_n\stackrel{P}{\longrightarrow} 
\ts,\quad\quad n\to\infty.\]
\end{thm}

To obtain a central limit theorem for the estimation error we consider finite and infinite activity separately, since we obtain different conditions on the relation of $n$, $\Delta_n$ and the cut-off sequence $v_n$.
% Let us first give the result for finite activity.

\begin{thm}[Asymptotic normality: finite activity]\label{thm:anfinact}
Assume that the L\'evy process $L$ has a finite jump activity : $\nu(\mathbb{R})<\infty$. Suppose that Assumptions \ref{ass:existence} to \ref{ass:ergodic}, \ref{ass:jumps}(i) and \ref{ass:ident} to \ref{ass:fischer} hold.

If  $n\Delta_n^{ 3-\eps}\to 0,$ $\sqrt n\Delta_n^{1-\eps/2}\left(\int_{|z|\leq 2v_n}\nu(dz)\right)^{1-\eps/2}\to 0$ and $\sqrt n \Delta_n^{1/2}\int_{|z|< {2v_n}}|z|\nu(dz)\to 0$ as $n\to\infty,$  then we conclude that
the FMLE $\hat{\t}_n$ is asymptotically normal: 
\[
t_{n}^{1/2}(\hat{\theta}_{n}-\ts)\stackrel{\L}{\to}N(0,I^{-1}(\ts)),\quad n\to\infty,
\]
 where $I$ is the Fisher information given by \eqref{eq:fisher}.

Furthermore, the FMLE $\hat \theta_n$ is asymptotically efficient in the sense of the Hàjek-Le Cam convolution theorem.
\end{thm}
\begin{rem}If $\nu$ has a bounded Lebesgue density, the conditions of the Theorem \ref{thm:anfinact} on the sampling scheme and the jump behavior reduce to $n\Delta_n^{3-4\eps}\to 0.$
\end{rem}
The following theorem generalizes the results of Theorem \ref{thm:anfinact} to driving L\'evy processes of infinite activity.

\begin{thm}[Asymptotic normality: general case]\label{thm:aninfinact}
Assume that the L\'evy process $L$ has infinite jump activity : $\nu(\mathbb{R})=\infty$.
Suppose Assumptions \ref{ass:existence} to \ref{ass:fischer} hold.
 If
$n\Delta_n^{3-\eps}\to 0,$ 
\[
\sqrt{n\Delta_{n}} \left(\int_{|z|\leq 3v_n/\gamma_{min}}|z|\nu(dz)\right )^{1-\eps/2}\to 0 \quad \text{and} \quad \sqrt n\Delta_n^{3/2-2\eps} \left(\int_{|z|\geq  3v_n/\gamma_{min}}\nu(dz)\right)^{1-\eps/2}\to0
\] 
 as $n\to\infty,$
then
all  conclusions of Theorem \ref{thm:anfinact} hold.
\end{thm}
Theorem \ref{thm:aninfinact} applies for both finite and infinite jump activity. Besides different conditions on the sampling scheme and the behavior of $\nu$ near zero it uses that the L\'evy measure $\nu$ admits a density, which is not supposed in Theorem \ref{thm:anfinact}. In the case where $\nu$ admits a bounded Lebesgue density, all the conditions on the $\Delta_n$ and $n$  of the Theorem \ref{thm:aninfinact} became $n\Delta_n^{3-\tilde\eps}\to 0$ for some $\tilde\eps>0$  as in the Theorem \ref{thm:anfinact}. 

\begin{example} [tempered stable jumps]
To illustrate the influence of the jump behavior of $L$ on the conditions on $n$ and $\Delta_n$ given in Theorem \ref{thm:aninfinact} let us consider the example of a tempered $\alpha$-stable driving L\'evy process. Tempered stable processes have been popular in financial modeling to overcome the limitations of the classical models based on Brownian motion alone (cf. \cite{Cont}). The Lévy measure in this case has an unbounded and non-integrable density given by
\[\nu(dz)=C|z|^{-(1+\alpha)}e^{-\lambda|z|}dz\]
with $ \lambda>0$ and a normalizing constant $C>0$ that satisfies the conditions of Theorem \ref{thm:aninfinact} if $0<\alpha<1$.   

The conditions on $n$, $\Delta_n$ and $\nu$ in Theorem \ref{thm:aninfinact} can now be summarized as $n \Delta_n^{2-\alpha -\tilde \epsilon} \to 0$ for some $\epsilon >0$. We observe that a higher Blumenthal-Getoor index $\alpha$ requires a faster convergence $\Delta_n$ to zero. This is in line with the intuition that when the intensity of small jumps increases (i.e. $\alpha$ increases) more and more frequent observations are needed to have a sufficient performance of the jump filter. 
\end{example}

%%%%%%%%%%%%%%%%%%%%%%%%%%%%%%%%%%%%%%%%%%%%%%%%%%%%%%%%%%%%%%%%%%%%%%
%%%%%%%%%%%%%%%%%%%%%%%%%%%%%%%%%%%%%%%%%%%
\section{Nonparametric estimation of $X^c$ via jump filtering. }\label{sec:jumpfilter}
%%%%%%%%%%%%%%%%%%%%%%%%%%%%%%%%%%%%%%%%%%%%%%
The estimation problem considered in this work leads naturally to the more fundamental problem of approximation of the continuous martingale part $X^c$ from discrete observations of a jump diffusion $X$.
 In this section we prove approximation results of this sort for integral functionals with respect to $X^c$. Since we need both uniform and non-uniform versions for the drift estimation problem, both settings will be discussed.
The following proposition concerns
 the finite activity case. The cut-off sequence $v_n$ and $\eps$ were defined in \eqref{eq:v_ni}.

%%%%%%%%%%%%%%%%%%%%%%%%%%%%

\begin{prop}[jump filtering: finite activity]
\label{prop:jumpfiltering}
Suppose that 
 $L$ is of finite activity and Assumptions \ref{ass:existence} to \ref{ass:jumps} hold.
 Suppose that  $f:\Theta\times \R \to \R$ satisfies:
 \begin{enumerate}
\item[a)] for all $x\in\R,$ $f(.,x)$ is Hölder continuous with respect to $\t\in\Theta:$
\[
 \forall \t, \t',
 \quad |f(\t,x)-f(\t',x)|\leq C(x)|\t-\t'|^{\kappa},\] where $0<\kappa\leq 1$  and $C:\R\to\R_+$ is at most of polynomial growth;

\item[b)] for all $ \t\in\Theta$, $f(\t,.)\in\C^2(\R)$ and $\sup_{\t\in\Theta}|f(\t,.)|$, $\sup_{\t\in\Theta}|f'_x(\t,.)|$ and $\sup_{\t\in\Theta} |f''_x(\t,.)|$ are at most of polynomial growth.

\end{enumerate} 
Then the following statements hold:
%\begin{itemize}
%\item for all $x\in\R,$ $f$ is Hölder continuous in $\t$,
% with Hölder exposant $\kappa\in]0,1]$ and 
%sub-polynomial Hölder constant $C(x)$; 
%\item for all $\t\in\Theta,$ $f (\theta, \cdot) \in\C^2(\R)$ and $\sup_{\t\in\Theta}|f(\t,x)|$, $\sup_{\t\in\Theta}|f'_x(\t,x)|$ and $\sup_{\t\in\Theta} |f''_x(\t,x)|$ are sub-polynomial. 
%\end{itemize}   

\begin{enumerate}
  
  \item [(i)] without any  assumption on the way that $\Delta_n\to 0$ as $n\to\infty,$
\begin{equation*}(n\Delta_n)^{-1}\sup_{\theta\in\Theta}\left|\int_{0}^{t_{n}}f(\theta,X_{s})\; dX_{s}^{c}-\sum_{i=1}^{n}f(\theta,X_{t_{i-1}})\Delta_{i}^nX\mathbf{1}_{|\Delta_{i}^nX|\leq v_n}\right|\stackrel{P}{\longrightarrow} 0;
\end{equation*}
%\label{eq:unif}

\item [(ii)] if $n\Delta_n^{3-\eps}\to 0,$ $\sqrt n\Delta_n^{1-\eps/2}\left(\int_{|z|\leq 2v_n}\nu(dz)\right)^{1-\eps/2}\to 0$ and \\
$\sqrt n \Delta_n^{1/2}\int_{|z|\leq {2v_n}}|z|\nu(dz)\to 0$ as $n\to\infty,$ then for any $\t\in\Theta$,

\begin{equation}\label{eq:jumpfiltapr}
 (n\Delta_n)^{-1/2}\left|\int_{0}^{t_{n}}f(\theta,X_{s})\; dX_{s}^{c}-\sum_{i=1}^{n}f(\theta,X_{t_{i-1}})\Delta_{i}^nX\mathbf{1}_{|\Delta_{i}^nX|\leq v_n}\right|\stackrel{P}{\longrightarrow} 0.
 \end{equation}

\end{enumerate}
\end{prop}

The case of infinite activity is treated in the following proposition.

\begin{prop}[jump filtering: infinite activity]
\label{prop:jumpfiltering_infinit}
Suppose that 
$L$ is of infinite activity and Assumptions \ref{ass:existence} to \ref{ass:jumps}  hold. 
Suppose that
$f:\Theta\times \R\to \R$ satisfies the assumptions of Proposition \ref{prop:jumpfiltering}. Then,
  \begin{enumerate}
\item [(i)] statement (i) of Proposition \ref{prop:jumpfiltering} holds;

\item[(ii)]
if
$n\Delta_n^{3-\eps}\to 0,$ 
\[
\sqrt{n\Delta_{n}} \left(\int_{|z|\leq 3v_n/\gamma_{min}}|z|\nu(dz)\right )^{1-\eps/2}\to 0 \quad and 
 \quad \sqrt n\Delta_n^{3/2-\eps} \left(\int_{|z|\geq  3v_n/\gamma_{min}}\nu(dz)\right)^{1-\eps/2}\to0
 \]
 as $n\to\infty,$
then for any $\t\in\Theta$, the convergence  \eqref{eq:jumpfiltapr} holds.

\end{enumerate}
\end{prop}

The proofs of both propositions are based on the following three lemmas. Lemma \ref{lem:filter_error_finite}
and \ref {lem:filter_infiniteactivity} describe the approximation of the discretized stochastic integral with respect to $X^c$ by the jump filter in the cases of finite and infinite activity, respectively. 
To prove the propositions \ref{prop:jumpfiltering} and \ref{prop:jumpfiltering_infinit} we also need a  convergence result for the Euler scheme in order to approximate the stochastic integral with respect to $X^c$ by the corresponding discrete sum. This will be done in Lemma \ref{lem:euler}.

%%%%%%%%%%%%%%%%%%%%%%%%%%
\begin{lem}[jump filtering error: finite activity]
\label{lem:filter_error_finite}

Assume that $L$ is of finite activity and $f:\Theta\times \R \to \R$ is such that $\sup_{\t\in\Theta}|f(\t,x)|$ is sub-polynomial.
Under Assumption \ref{ass:existence} to \ref{ass:jumps},  we obtain
\begin{enumerate}
\item[(i)]
 \begin{equation*}
 \sup_{\t\in\Theta}| \sum_{i=1}^{n}f(\theta,X_{t_{i-1}})\left(\Delta_{i}^nX^{c}-\Delta_{i}^nX\mathbf{1}_{|\Delta_{i}^nX|\leq v_n}\right)|=O_{ \LL^1}( n\Delta_n^{3/2-\eps/2}).
 \end{equation*}
\item [(ii)] for all $\t\in\Theta,$ if $n\Delta_{n}^{3-\eps}\to0 $ as $n\to\infty$,
\begin{align*}
 & \sum_{i=1}^{n}f(\theta,X_{t_{i-1}})\left(\Delta_{i}^nX^{c}-\Delta_{i}^nX\mathbf{1}_{|\Delta_{i}^nX|\leq v_n}\right)=o_{P}\left(\sqrt{n\Delta_{n}}\right)\\
 & +O_{\LL^1}\left(n\Delta_n^{5/2-\eps}+n\Delta_n^{3/2-\eps/2}\left(\int_{|z|\leq 2v_n}\nu(dz)\right )^{1-\eps/2}+n\Delta_{n}\int_{|z|\leq 2v_n}|z|\nu(dz)\right).
 \end{align*}

\end{enumerate}

\end{lem}

The next lemma extends the uniform bound to the case of infinite activity.

\begin{lem}[jump filtering error: infinite activity]
\label{lem:filter_infiniteactivity}
Assume that $L$ is of infinite activity and $f:\Theta\times \R \to \R$ is such that $\sup_{\t\in\Theta}|f(\t,x)|$ is sub-polynomial.
\begin{enumerate}
\item[(i)]
 Under Assumption \ref{ass:existence} to \ref{ass:jumps}, we obtain

\begin{align*}
 &\sup_{\t\in\Theta}|\sum_{i=1}^{n}f(\theta,X_{t_{i-1}})\left(\Delta_{i}^nX^{c}-\Delta_{i}^nX\mathbf{1}_{|\Delta_{i}^nX|\leq v_n}\right)|=\\
 & O_{\LL^1}\left({n\Delta_{n}} \left(\int_{|z|\leq 3v_n}|z|\nu(dz)\right )^{1-\eps/2}+
  n\Delta_n^{3/2-\eps} \left(\int_{|z|\geq v_n/\gamma_{min}}\nu(dz)\right)^{1-\eps/2}
\right)
\end{align*}
\item [(ii)]  for all $\t\in\Theta,$ if $n\Delta_n^{3-\eps}\left(\int_{|z|\geq 3v_n/\gamma_{min}}\nu(dz)\right)^{2-\eps}\to 0,$ 
 as $n\to\infty,$  then
\begin{align*}
 & \sum_{i=1}^{n}f(\theta,X_{t_{i-1}})\left(\Delta_{i}^nX^{c}-\Delta_{i}^nX\mathbf{1}_{|\Delta_{i}^nX|\leq v_n}\right)= o_{P}\left(\sqrt{n\Delta_{n}}\right)\\ &+o_{\LL^1}\left(n\Delta_n^{2-\eps}(\int_{|z|\geq 3v_n/\gamma_{min}}\nu(dz))^{1-\eps/2}\right)
  +O_{\LL^1}\left(n\Delta_n(\int_{|z|\leq 3v_n/\gamma_{min}}|z|\nu(dz))^{1-\eps/2}\right ). 
 \end{align*}

\end{enumerate}

\end{lem}

The approximation of the stochastic integral is treated in the following
lemma.

%%%%%%%%%%%%%%%%%%%%%
\begin{lem}[Euler scheme]\label{lem:euler} 

Suppose that  $f:\Theta\times \R \to \R$ satisfies the following assumptions:
\begin{enumerate}
\item[a)] for all $x\in\R,$ $f(.,x)$ is Hölder continuous with respect to $\t\in\Theta:$
\[\forall  \t, \t',  \quad |f(\t,x)-f(\t',x)|\leq K(x)|\t-\t'|^{\kappa};\] where $0<\kappa\leq 1$  and $K:\R\to\R_+$ is at most of polynomial growth;

\item[b)] for all $ \t\in\Theta$, $f(\t,.)\in\C^2(\R)$ and $\sup_{\t\in\Theta}|f(\t,.)|$, $\sup_{\t\in\Theta}|f'_x(\t,.)|$ and $\sup_{\t\in\Theta} |f''_x(\t,.)|$ are at most of polynomial growth.

\end{enumerate} 
Under Assumptions \ref{ass:existence} to \ref{ass:jumps}, we obtain

\begin{enumerate}
\item [(i)] as $n\to\infty,$
\[
\sup_{\t\in\Theta}(n\Delta_n)^{-1} \left|\int_{0}^{t_{n}}f(\theta,X_{s})\; dX_{s}^{c}-\sum_{i=1}^{n}f(\theta,X_{t_{i-1}})\Delta_{i}^nX^{c} \right|\stackrel{P}{\longrightarrow} 0;
\]

\item [(ii)]
if $n\Delta_n^{3-\eps}\to 0,$ then, as $n\to\infty,$
\[\forall \t\in\Theta,\quad
(n\Delta_n)^{-1/2} \left|\int_{0}^{t_{n}}f(\theta,X_{s})\; dX_{s}^{c}-\sum_{i=1}^{n}f(\theta,X_{t_{i-1}})\Delta_{i}^nX^{c} \right|\stackrel{P}{\longrightarrow} 0.\]

\end{enumerate}
\end{lem}

We have now collected all the tools to prove the convergence of the jump filter approximation towards integral functionals with respect to the continuous martingale part as stated in Proposition \ref{prop:jumpfiltering} and \ref{prop:jumpfiltering_infinit}.

\begin{proof}[Proof of Proposition \ref{prop:jumpfiltering}]
We decompose the difference as follows:
\begin{multline}
 \left|\int_{0}^{t_{n}}f(\theta,X_{s})\; dX_{s}^{c}-\sum_{i=1}^{n}f(\theta,X_{t_{i-1}})\Delta_{i}^nX\mathbf{1}_{|\Delta_{i}^nX|\leq v_n}\right|\leq\\
 \left|\int_{0}^{t_{n}}f(\theta,X_{s})\; dX_{s}^{c}-\sum_{i=1}^{n}f(\theta,X_{t_{i-1}})\Delta_{i}^nX^{c}\right|+\left|\sum_{i=1}^{n}f(\theta,X_{t_{i-1}})\Delta_{i}^nX^{c}-\sum_{i=1}^{n}f(\theta,X_{t_{i-1}})\Delta_{i}^nX\mathbf{1}_{|\Delta_{i}^nX|\leq v_n}\right|
 \label{al:ddcomp}
\end{multline}
We first prove (i). By Lemma \ref{lem:euler},
the first term on the right hand side of \eqref{al:ddcomp} divided by $n\Delta_n$ goes to zero uniformly, 
 without any condition on $\Delta_n.$ Combining it with (i) of the Lemma \ref{lem:filter_error_finite} we get the result.
%\[
%{n\Delta_n}^{-1}\sup_{\theta\in\Theta}\left|\int_{0}^{t_{n}}f(\theta,X_{s})\; dX_{s}^{c}-\sum_{i=1}^{n}f(\theta,X_{t_{i-1}})\Delta_{i}^nX^{c}\right|\stackrel{P}{\longrightarrow} 0.\]
%ICICICI

We now prove (ii). 
 For the first term of \eqref{al:ddcomp} divided by $(n\Delta_n)^{1/2}$ we use (ii) of the Lemma~ \ref{lem:euler}. 
 %Namely,
%for all $\t\in\Theta,$ if $\exists \eta >0$, such that $n\Delta_n^{3-\eta}\to 0,$ then
%\[
%(n\Delta_n)^{-1/2}|\int_{0}^{t_{n}}f(\theta,X_{s})\; dX_{s}^{c}-\sum_{i=1}^{n}f(\theta,X_{t_{i-1}})\Delta_{i}^nX^{c}|\stackrel{P}{\longrightarrow} 0.\]
%Du to the condition (i), we can take $\eta=\eps.$
Moreover, 
the (ii) of the Lemma \ref{lem:filter_error_finite},
 gives, for any $\t\in\Theta,$
\begin{align*}
 & (n\Delta_n)^{-1/2}\sum_{i=1}^{n}f(\theta,X_{t_{i-1}})\left(\Delta_{i}^nX^{c}-\Delta_{i}^nX\mathbf{1}_{|\Delta_{i}^nX|\leq v_n}\right)=o_{P}(1)+\\
 & O_{\LL^1}\left(\sqrt n\Delta_{n}^{2-\eps}+ \sqrt n\Delta_n^{1-\eps/2}\left (\int_{|z|\leq 2 v_{n}}\nu(dz)\right)^{1-\eps/2}+\sqrt n\Delta_{n}^{1/2}\int_{|z|\leq 2v_{n}}|z|\nu(dz)\right)\stackrel{P}{\longrightarrow} 0
\end{align*}
under conditions (ii) of the proposition.

\end{proof}

\begin{proof}[Proof of Proposition \ref{prop:jumpfiltering_infinit}]
We use the decomposition \eqref{al:ddcomp}

 and prove first the statement (i).
Using the Lemma \ref{lem:euler} the first term of \eqref{al:ddcomp} divided by $n\Delta_n$ goes to zero uniformly without any condition on $\Delta_n.$

Lemma \ref{lem:filter_infiniteactivity} together with the Assumption \ref{ass:jumps} (ii) and the fact that $v_n=\Delta_n^{1/2-\eps}$ gives
\begin{align*}
 & (n\Delta_n)^{-1}\sup_{\theta}|\sum_{i=1}^{n}f(\theta,X_{t_{i-1}})\left(\Delta_{i}^nX^{c}-\Delta_{i}^nX\mathbf{1}_{|\Delta_{i}X|\leq v_n}\right)|=\\
 &O_{\LL^1}\left( \left(\int_{|z|\leq 3v_n}|z|\nu(dz)\right )^{1-\eps/2}+
  \Delta_n^{1/2-\eps/2} \left(\int_{|z|\geq  v_n/\gamma_{min}}\nu(dz)\right)^{1-\eps/2}
\right)=\\
&O_{\LL^1}\left( \left(\int_{|z|\leq 3v_n}|z|\nu(dz)\right )^{1-\eps/2}+
  \frac{\Delta_n^{1/2-\eps/2}} {v_n^{1-\eps/2}}\left(\int_{|z|\geq  v_n/\gamma_{min}}{|z|}\nu(dz)\right)^{1-\eps/2}
\right)\stackrel{P}{\longrightarrow} 0.
\end{align*}
Hence statement (i) is proved. \\
Now we prove statement (ii).
For any $\t\in\Theta,$ under the condition $n\Delta_n^{3-\eps}\to 0,$ the second statement of Lemma \ref{lem:euler} gives the convergence to $0$ of the first term in the decomposition \eqref{al:ddcomp}, divided by $\sqrt n\Delta_n$.  The convergence to $0$ of the second term of \eqref{al:ddcomp}, divided by $\sqrt n\Delta_n$, immediately follows from Lemma \ref{lem:filter_infiniteactivity} and the conditions of (ii).
\end{proof}
When discretizing the likelihood function, we need the following lemma, whose proof can be found in the Section \ref{sec:auxlemmas}.

\begin{lem}
\label{lem:Riemann_app}
 Suppose that Assumptions \ref{ass:existence}--\ref{ass:jumps} are satisfied . Suppose that 
 $f:\Theta\times\R\to\R$ is such that $\forall \t\in\Theta$, $f(\t,.)\in\C^1(\R)$ and $\sup_{\t\in\Theta}|f'(\t,.)|$ is  sub-polynomial. 
 Then we obtain:
 \begin{enumerate}
 \item[
 (i)] as $n\to\infty$,
\[
\sup_{\theta\in\Theta}\left|\int_{0}^{t_{n}}f(\theta,X_{s})\; ds-\sum_{i=1}^{n}f(\theta,X_{t_{i-1}})\Delta_{i}^nId\right|=O_{\LL^1}(
 n 
\Delta_{n}^{3/2});
\]

\item[(ii)] if $n\Delta_n^{3-\eps} \xrightarrow{n \to \infty} 0$,  then
\begin{equation*}
	(n\Delta_n)^{-1/2} |\int_{0}^{t_{n}}f(\theta,X_{s})\; ds-\sum_{i=1}^{n}f(\theta,X_{t_{i-1}})\Delta_{i}^nId|\xrightarrow{P} 0
	.	
\end{equation*}
\end{enumerate}
\end{lem}

\section{Examples and numerical results} \label{sec:applications}

In this section we consider concrete applications of the drift estimator in popular jump diffusion models and investigate the numerical performance in finite sample studies. We consider both examples with finite and infinite jump activity. 

In the first part we give explicit drift estimators for Ornstein-Uhlenbeck-type and CIR processes and compare there performance in a Monte Carlo study for finite activity jumps. Then we apply our method to a hyperbolic diffusion process with $\alpha$-stable jump component of infinite jump activity. We consider here for convenience only linear models in the drift parameter that lead to explicit maximum likelihood estimators in order to avoid the need for numerical maximization techniques. Note that the method developed in this work applies equally well to non-linear models by using standard maximization methods on the discretized and jump-filtered likelihood function \eqref{eq:like_filtered}.

It turns out that our estimators can be applied even beyond the scope of our theoretical results. To demonstrate this we include in Section \ref{sec:ex_inf} models that do not posses moments of all orders and consider L\'evy processes of unbounded variation in our simulations.

\subsection{Finite activity}

In this section we consider two different jump diffusion models with finite activity jumps. The first model will consist of Ornstein-Uhlenbeck-type processes that recently became popular in financial modeling (cf. for example \cite{BNS01}). In the second part we extend a Cox-Ingersoll-Ross model from finance (cf. \cite{CIR}) by including jumps and investigate the finite sample behavior of the drift estimator and jump filter for varying observation settings. The jump process $L$ is of compound Poisson type in the case of finite activity such that it can be written as
\begin{equation} \label{eq:compPois}
L_t=\sum_{i=1}^{N_t}Z_i, \text{ for $t \geq 0$,}
\end{equation} 
where $(N_t)_{t\geq 0}$ is a Poisson process with intensity $\lambda$ and $(Z_{i})_{i\in\N}$ are i.i.d. real random variables independent of $N$, with distribution $\nu/\lambda$.

\subsubsection{Ornstein-Uhlenbeck-type processes}

Suppose that we have given a discrete sample 
\begin{equation} \label{eq:observ_scheme}
 X_{t_0},\ldots,X_{t_n}\quad \text{ for $t_i = i \Delta_n$ and $i=0,\ldots,n$, }
\end{equation}
of an Ornstein-Uhlenbeck-type (OU) process $(X_t)_{t\geq 0}$ that is defined as a solution of the stochastic differential equation
\[
 dX_t =(\theta_2 - \theta_1 X_t) \; dt + \sigma\; dW_t + dL_t \quad X_0 = x,
\]
where $(W_t)_{t\geq 0}$ is a standard Brownian motion and $(L_t)_{t \geq 0}$ a pure jump L\'evy process. Our goal is to estimate the unknown drift parameter $\theta= (\theta_1,\theta_2) \in \mathbb{R}^2$. The volatility parameter $\sigma>0$ might be unknown and can be seen as a nuisance parameter. The jump component $(L_t)_{t\geq 0}$ will be of compound Poisson type, i.e. it can be written as in \eqref{eq:compPois}
 with intensity $\lambda$ and the jump heights $Z_i$ are supposed to be iid with exponential distribution with rate $1$. 

From \eqref{eq:like_filtered} and \eqref{eq:FMLE} we find that the FMLE for $\theta$ is the solution $\hat \theta_n^{\text{OU}} =(\hat \theta_{1,n}^{\text{OU}}, \hat \theta_{2,n}^{\text{OU}})$ to the following set of linear equations in $\theta_1$ and $\theta_2$.
\begin{align} 
   \theta_1 &=  \frac{\theta_2 I_n (X,1) - \sum_{i=1}^n X_{t_i} \Delta_i^n X \mathbf{1}_{|\Delta_i^n X| \leq v_n}}{I_n (X^2)}, \nonumber \\
   \theta_2 &=  \frac{\sum_{i=1}^n  \Delta_i^n X \mathbf{1}_{|\Delta_i^n X| \leq v_n}+\theta_1 I_n (X,1)}{t_n}, \label{eq:OU_FMLE2}
\end{align}
where we introduced the functional
\begin{equation}\label{Inp}
 I_n (X,p) := \sum_{i=1}^n X_{t_i}^p \Delta_i^nId \quad \text{for $p \in \mathbb{R}$}.
\end{equation}
The FLME for the first component of $\theta$ results in
\[
 \hat \theta_{1,n}^{\text{OU}} = \left( 1- \frac{I_n (X,1)^2}{I_n (X,2)} \right)^{-1} \frac{ I_n (X,1) \sum_{i=1}^n  \Delta_i^n X \mathbf{1}_{|\Delta_i^n X| \leq v_n}-t_n \sum_{i=1}^n X_{t_i} \Delta_i^n X \mathbf{1}_{|\Delta_i^n X| \leq v_n}}{t_n I_n (X,2)}.
\]
The second component $\hat \theta_{2,n}^{\text{OU}}$ follows now easily by plugging $\hat \theta_{1,n}^{\text{OU}}$ into \eqref{eq:OU_FMLE2}.

In Table \ref{tableOU} we give simulation results for $\hat \theta_{1,n}^{\text{OU}}$. The given mean and standard deviation are each based on 500 Monte Carlo samples of $\hat \theta_{1,n}^{\text{OU}}$. In this example we choose $v_n = \Delta_n^{1/3}$ in order to approximate well the continuous martingale part that appeared in the likelihood function \eqref{eq:likelihood}. We compare different observation schemes and different jump intensities $\lambda$ for true parameter values given by $\theta_1 = 2$ and $\theta_2 = 0$. The drift estimator performs well over the whole range of settings provide that the discretization distance $\Delta_n$ is sufficiently small. We also give the average number of jumps that were detected by the jump filter and observe that this number scales as expected linearly in $t_n$.

\begin{table} 
\begin{tabular}{ c c c c c c c c }
 \multicolumn{2}{c}{} & \multicolumn{3}{c}{$\lambda$ = 1} & \multicolumn{3}{c}{$\lambda$ = 6} \\
 \toprule
  $t_n$ & $n$ & mean & std dev & jumps filt & mean & std dev & jumps filt\\
\midrule
 
 2 & 100 & 1.4 & 0.7  & 6.5& 1.4 &0.6 & 15.8 \\
 
  & 300 & 1.8 & 0.8  & 6.8& 1.7 &0.6 & 15.9 \\
 
  & 600 & 2.0 & 0.8  & 7.9 & 1.9 & 0.5 & 16.3   \\

   & 800 & 2.0 & 0.8 & 7.2 & 2.0 & 0.6 & 16.5  \\

 5 & 600 & 1.4 & 0.6  & 13.1 &1.3 & 0.39 & 39.5 \\
 
 & 1200 & 1.8 & 0.6  & 13.6 &1.7 & 0.39 & 40.4 \\
 
 & 4000 & 2.0 & 0.7  & 13.6 &1.8 & 0.39 & 41.4 \\

    & 6000 & 2.1 & 0.7  & 12.4 & 1.9 & 0.37 & 41.5  \\

  10 & 600 & 1.2 & 0.26  & 19.1 &1.3 & 0.21 & 67 \\
 
 & 2000 & 2.0 & 0.27  & 21.6 &1.6 & 0.2 & 75 \\

\midrule
\end{tabular}
\caption{Monte Carlo estimates of mean and standard deviation from 500 samples of $\hat{\theta}_{1,n}^{\text{OU}}$ for an OU process with compound Poisson jumps with intensity $\lambda$ and true parameter $\theta_1 =2$.}\label{tableOU}
\end{table}

% \begin{figure}
% \centering
% \includegraphics[scale=0.5]{images/standdev}
% \caption{Estimated standard deviation of $\hat{\theta}_n^{\text{OU}}$ from 500 Monte Carlo samples for observation distance $\Delta_n = 0.0025$ and true parameter $\theta = 2$} \label{figure:barOU}
% \end{figure}

\subsubsection{Cox-Ingersoll-Ross (CIR) processes with jumps}

We define a CIR or square-root process $X=(X_t)_{t \geq 0}$ with jumps as a solution to the SDE
\[
 X_t = (\theta_1 - \theta_2 X_t) \; dt + \sigma \sqrt{X_t} \; dW_t + dL_t,
\] 
where $\theta_1, \theta_2, \sigma > 0$, $(W_t)$ is a standard Brownian motion and $(L_t)$ a pure jump L\'evy process. The two-dimensional drift parameter $\theta = (\theta_1, \theta _2)$ is unknown and will be estimated from discrete observations of $X$ as in \eqref{eq:observ_scheme}.

The classical CIR process without jumps (e.g. $L_t \equiv 0$) has the property that it stays non-negative at all times which makes it an interesting model for financial applications e.g. in interest rate modeling (Vasicek model) and stochastic volatility models (Heston model). We consider here therefore a jump component $(L_t)$ of compound Poisson type that exhibits only positive jumps such that $X$ will stay non-negative. In fact, we take a driving L\'evy process $(L_t)$ as in \eqref{eq:compPois} with intensity $\lambda = 1$ and  exponentially distributed jumps with rate $\eta>0$, e.g. $Z_i \sim \text{Exp}(\eta)$.

The filtered maximum likelihood estimator for $\theta$ is this model can be easily derived from \eqref{eq:like_filtered} and \eqref{eq:FMLE}. It is given as the solution $\hat \theta_n^{\text{CIR}} =(\hat \theta_{1,n}^{\text{CIR}}, \hat \theta_{2,n}^{\text{CIR}})$ to the following set of linear equations in the parameters $\theta_1$ and $\theta_2$.
\begin{align} \label{eq:lin_theta}
  \theta_1 = \frac{\theta_2 t_n - \sum_{i=1}^n X_{t_i}^{-1} \Delta_i^n X \mathbf{1}_{|\Delta_i^n X| \leq v_n} }{I_n (X,{-1})}, \quad
 \theta_2 = \frac{\theta_1 t_n - \sum_{i=1}^n \Delta_i^n X \mathbf{1}_{|\Delta_i^n X| \leq v_n} }{I_n (X,1)},
\end{align}
where $I_n (X,p)$ for $p \in \mathbb{R}$ was defined in \eqref{Inp}.
We obtain for $\hat \theta_{2,n}^{\text{CIR}}$ the FMLE
\[
 \hat \theta_{2,n}^{\text{CIR}} = \left( I_n (X,{-1}) I_n (X,1) - t_n \right)^{-1} \left(\sum_{i=1}^n X_{t_i}^{-1} \Delta_i^n X \mathbf{1}_{|\Delta_i^n X| \leq v_n} - I_n (X,{-1}) \sum_{i=1}^n \Delta_i^n X \mathbf{1}_{|\Delta_i^n X| \leq v_n} \right).
\]
The first component $\hat \theta_{1,n}^{\text{CIR}}$ follows now immediately by plugging $\hat \theta_{2,n}^{\text{CIR}}$ into \eqref{eq:lin_theta}.

To obtain Monte Carlo estimates of mean and standard deviation of $\hat \theta_n^{\text{CIR}}$ we simulate discrete samples of $X$ on an equidistant grid as in the previous example. We take $v_n = \Delta_n^{1/3}$ in order to approximate the continuous martingale part of $X$. In Table \ref{table:CIR} we report the results for $\hat \theta_{2,n}^{\text{CIR}}$ from 1000 Monte Carlo samples each. The results are given for different $t_n$, $n$ and $\sigma$ for true parameter values $\theta_1 =0.1$ and $\theta_2=2$. We find that $\hat \theta_{2,n}^{\text{CIR}}$ performs well as long as the discretization step size $\Delta_n$ is fine enough such that a high-frequency approximation becomes valid.

\begin{table} 
\begin{tabular}{ c c c c c c c c }
 \multicolumn{2}{c}{} & \multicolumn{3}{c}{$\sigma$ = 0.25} & \multicolumn{3}{c}{$\sigma$ = 0.5} \\
 \toprule
  $t_n$ & $n$ & mean & std dev & jumps filt & mean & std dev & jumps filt\\
\midrule
  5 & 200 &1.7&  0.22 & 6.8& 1.7 &0.28 &8.0 \\
  
  & 400 &1.9&  0.12 & 5.1& 1.8 &0.2 &6.6 \\
  
  & 800 &2.0&  0.09 & 4.5& 1.9 &0.17 &5.6 \\

 10 & 500 & 1.7 & 0.15  & 12 & 1.7 & 0.21 & 15   \\
 
 & 1000 & 1.9 & 0.08  & 9.7 & 1.8 & 0.14 & 12   \\

   & 1500 & 1.9 & 0.06 & 9.5 & 1.9 & 0.13 & 11\\

 20 & 1000 & 1.8 & 0.13  & 25 & 1.6 & 0.16 & 30 \\
 
 & 2000 & 1.9 & 0.06  & 19 & 1.8 & 0.11 & 24 \\

    & 3000 & 2.0 & 0.04  & 19 & 1.9 & 0.09 & 22  \\

\midrule
\end{tabular}
\caption{Monte Carlo estimates of mean and standard deviation of $\hat{\theta}_{2,n}^{\text{CIR}}$ for a CIR process with Gaussian component and compound Poisson jumps with intensity $\lambda=1$ and true drift parameter $\theta_2 = 2$.}\label{table:CIR}
\end{table}

\subsection{Infinite activity}\label{sec:ex_inf}

In this section we investigate estimation of the drift when the driving L\'evy process is of infinite jump activity. This is of course a more challenging problem with regards to the approximation of the continuous martingale part i.e. the jump filtering problem, since we have to distinguish a diffusion component from a process that jumps infinitely often in finite time intervals.

\subsubsection{Hyperbolic diffusions with jumps}

In this section we apply the drift estimator to hyperbolic diffusion processes with jumps. They are defined as solutions $(X_t)_{t\geq 0}$ of the following SDE:
\[
 dX_t = -\frac{\theta X_t}{(1+ X_t^2)^{1/2}} \; dt + \sigma \; dW_t + dL_t, \quad X_0 = x.
\]
Here, the drift parameter $\theta > 0$ and the diffusion coefficient $\sigma >0$ are unknown and we aim at estimating $\theta$ form discrete observations $X_{t_0},\ldots, X_{t_n}$ of $X$, where $t_i = i \Delta_n$ for $\Delta_n >0$ and $i=0,\ldots, n$. The driving L\'evy process $(L_t)_{t\geq 0}$ will be an $\alpha$-stable process with L\'evy-Khintchine triplet $(0,0,\nu)$ such that the L\'evy measure is of the form $\nu(dx) = dx/|x|^{1+\alpha}$.

From \eqref{eq:likelihood} we obtain an explicit pseudo MLE for $\theta$ in this model class given by
\[
 \hat \theta_t^{\text{hyp}} = -\frac{\int_0^t \frac{X_s}{(1+X_s^2)^{1/2}}dX_s^c}{\int_0^t \frac{X_s^2}{(1+X_s^2)}ds}.
\]
Via discretization and jump filtering this leads to the following drift estimator based on discrete observations:
\[
 \hat \theta_n^{\text{hyp}} = -\sum_{i=1}^n \frac{X_{t_i}}{(1+ X_{t_i}^2)^{1/2}}\Delta_i^n X \mathbf{1}_{|\Delta_i^n X| \leq v_n} \left(\sum_{i=1}^n \frac{X_{t_i}^2}{(1+ X_{t_i}^2)}  \right)^{-1}
\]
To assess the performance of $\hat \theta_n^{\text{hyp}}$ in Monte Carlo experiments we simulate discrete trajectories of $X$ via a Euler scheme with sufficient small step size. 

In Table \ref{table:HD} we give estimated mean and standard deviation of $\hat \theta_n^{hyp}$ from 500 Monte Carlo samples each for different observation length $t_n$ and number of observations $n$. We consider two different values for the index of stability $\alpha$ and give also the number of jumps that have been detected by the jump filter. It turns out that $\hat \theta_n^{hyp}$ performs remarkably well over the whole range of different setting even in the case $\alpha = 1$ of infinite variation jumps that is not covered by our theoretical results, since we have assumed that $\int_{\mathbb{R}} |x| \nu (dx) < \infty$. It might therefore be reasonable to expect that the convergence results presented here can be extended to jumps processes with Blumenthal-Getoor index $\alpha \geq 1$. 
\begin{table} 
\begin{tabular}{ c c c c c c c c }
 \multicolumn{2}{c}{} & \multicolumn{3}{c}{$\alpha$ = 0.5} & \multicolumn{3}{c}{$\alpha$ = 1} \\
 \toprule
  $t_n$ & $n$ & mean & std dev & jumps filt & mean & std dev & jumps filt\\
\midrule
  5 & 600 &1.7& 0.53  & 26& 1.6 &0.62 &37 \\
  
  & 1200 &1.9& 0.54  & 27& 1.8 &0.60 &40 \\

    & 1500 & 1.9& 0.57 & 26 & 1.9 & 0.66 &41  \\

 10 & 1000 &1.6 & 0.33 & 51 & 1.5 & 0.40 & 71 \\
 
 & 2000 & 1.8 & 0.34 & 53 & 1.7 & 0.38 & 79 \\
 
 & 4000 & 1.9 & 0.35  & 50 &1.9 & 0.43  & 85   \\

 20 & 2000 & 1.6 & 0.23  & 104 & 1.6 & 0.27 & 142  \\
 
 & 4000 & 1.8 & 0.24  & 106 & 1.7 & 0.28 & 158  \\
 
 & 8000 & 1.9 & 0.23  & 101 & 1.9 & 0.30 & 170 \\

\midrule
\end{tabular}
\caption{Monte Carlo estimates of mean and standard deviation from 500 samples of $\hat{\theta}_n^{\text{hyp}}$ for a hyperbolic diffusion process with Gaussian component and $\alpha$-stable jumps and true drift parameter $\theta = 2$}\label{table:HD}
\end{table}

%%%%%%%%%%%%%%%%%%%%%%%%%%%%%%%%%%%
\section{Proofs of main results}\label{sec:mainproofs}
%%%%%%%%%%%%%%%%%%%%%%%%%%%%%%%%%%%%%%
\subsection{MLE for continuous observations}
Let $\htet $
be the true MLE maximizing
 the log-likelihood function  given by \eqref{eq:loglikelihood} and
based on continuous observations  :
\begin{equation}\label{eq:estimator} 
\htet \in \argmax_{\theta\in\Theta}\ell_t(\theta).
\end{equation}
%In this section we give the likelihood function in (\ref{eq:Xsde})
%based on continuous observations, discuss its asymptotic behavior
%and deduce the asymptotic behavior of the likelihood estimator.
%%%%%%%%%%%%%%%%%%%%%%%%%%%
Before moving to discrete observations we prove here some asymptotic
results for $\htet$. This is a first step in order to prove the asymptotic results for the
FMLE.
%the true MLE derived by maximizing (\ref{eq:likelihood})
%based on continuous observations. 
%These asymptotic properties will
%serve us later on as a benchmark for the case of discrete observations.
%Let us denote by $l_{t}(\theta,X)=\log L_{t}(\theta,X)$ the log-likelihood
%function.
% and by $\partial_{\theta}\ell_{t}(\theta,X)$ its derivative
%with respect to $\theta$.
%We assume that the process $X$ is generated under the true parameter value $\ts,$ an interior point of a parameter space $\mathring {\Theta}.$
%We shorten $X$ for $X^{\ts},$ $\pi$ for $\pi^{\ts}$ and $P$ for $\PP^{\ts}.$
%A MLE $ \htet$ of $ \ts $
%is defined as a measurable choice
%%

%%% the polynomial growth provide that $\int_0^t\frac{\dot b}{\sigma}dW_s$ is a square integrable martingale and permit apply TCL
%%%

%\begin{assumption}(Integrability w.r.t. $\pi$)\label{ass:integrab}
%For all $\t\in\Theta,$ the functions $ {\dot b(\t,.)}$, ${\ddot b(\t,.)}$ 
%are of polynomial growth.
%  \end{assumption}

%%%%%%%%%%%%%%
 
%%%%%%%%%%%%%%%
%%% CONSISTENCY OF FMLE %%%
%%%%%%%%%%%%%%

\begin{thm} \label{th: FMLE_consist}
Suppose that Assumptions \ref{ass:existence}--\ref{ass:ident} and \ref{ass:hoelder}(i) are satisfied. Then 
\[
\lim_{t\to\infty} \htet =\ts 
 \quad P-a.s.
 \] 
\end{thm}
%%
%
%The proof of Theorem~\ref{th: FMLE_consist} relies on classical theory
%about the convergence of  maximum likelihood estimators, as stated for
%instance in the classical approach by~\cite{Wald} for i.i.d.  random
%variables. 
%We  refer for  instance to  Theorem  5.7 in~\cite{Vaart}  for a  simple
%presentation of Wald's  approach and further stress that  the proof is
%valid on a compact 
%parameter space only. 
%It relies on the two following ingredients.
%%%
%\begin{prop}\label{prop:LimitEllTheta}
%Under  Assumptions~\ref{ass:ergodic},~\ref{ass:hölder}, there
%exists a finite deterministic limit $\ell(\t)$ such that for all compact $K\in\Theta$,  
%%%
%\begin{equation{\star}}
%  \sup_{\t\in K}
%  |\frac 1t \ell_t(\t) - \ell(\t)|\xrightarrow[n \to \infty]{}0 \quad
%\mbox{in }  
%P \mbox{a.s.}.
%\end{equation{\star}}
%%%
%\end{prop}
%%%
%
%\begin{prop} \label{prop:IdentificationEllTheta}
%Under Assumptions~\ref{ass:ergodic} to~\ref{ass:hölder}, for any
%$\eps >0$, 
%%%
%\[
%\sup_{\theta: \|\theta-\ts\| \geq \eps} \ell(\theta)<\ell(\ts).
%\]
%\end{prop}
%%%
%Proposition~\ref{prop:LimitEllTheta}  induces a  uniform convergence
%of the normalised criterion~$\ell_t/t$ 
%to the limiting function $\ell$.
%  Proposition~\ref{prop:IdentificationEllTheta}
%states that  the former limiting  function $\ell$ identifies  the true
%value of the 
%parameter   $\ts$,  as  the   unique  point   where  it   attains  its
%maximum.     
%

\begin{proof}
%{\it (of the Proposition~\ref{prop:LimitEllTheta})}
%\begin{equation}
% \ell_t(\t,X)=-\frac 12\int_0^t\frac{(b(\t,X_s)-b(\ts,X_s))^2}{\sigma^2(X_s)}ds+\frac 12\int_0^t\frac{b^2(\ts,X_s)}{\sigma^2(X_s)}ds+\int_0^t\frac{b(\t,X_s)}{\sigma(X_s)}dW_s.
%\end{equation}
Denote
\begin{equation}\label{eq:ellt}
\tilde\ell_t(\t):=\int_0^t\frac{(b(\t,X_s)-b(\ts,X_s))}{\sigma(X_s)}dW_s-\frac 12
\int_0^t\frac{(b(\t,X_s)-b(\ts,X_s))^2}{\sigma^2(X_s)}ds.
\end{equation}

Using \eqref{eq:Xsde} and the fact that the observed trajectory  corresponds to the true value of parameter $\ts$, we can easily see that

\[\ell_t(\t)=\tilde\ell_t(\t)+\frac 12\int_0^t\frac{b(\ts,X_s)}{\sigma(X_s)}dW_s+\frac12\int_0^t\frac{b^2(\ts,X_s)}{\sigma^2(X_s)}ds.\]
The difference between $\ell(\t)$ and $\tilde\ell_t(\t)$ does not depend on $\t$, hence also
\begin{equation}
\htet \in \argmax_{\theta\in\Theta}\tilde\ell_t(\theta).
\end{equation}
For $\t\in\Theta,$ define 
\[M_t(\t):=\int_0^t\frac{(b(\t,X_s)-b(\ts,X_s))}{\sigma(X_s)}dW_s.\]
The process $(M_t(\t), t\geq 0)$ is a continuous local martingale, with quadratic variation given by
\[
 A_t(\t):=<M(\t)>_t=\int_0^t\frac{(b(\t,X_s)-b(\ts,X_s))^2}{\sigma^2(X_s)}ds.\]
Note that
 \begin{equation}\label{eq: ell_t}
\tilde\ell_t(\t)=-\frac 12 A_t(\t)+M_t(\t),
\end{equation}

Recall that $\pi$, given by the Lemma \ref {lem:ergodic} is an invariant distribution of $X$ and denote 
\begin{equation}\label{eq:ell}
\tilde\ell(\t)=-\frac{1}{2}\pi\left (\frac{(b(\t,.)-b(\ts,.))^2}{\sigma^2(.)}\right)
\end{equation}
Using Assumptions \ref{ass:girsanov}, \ref{ass:hoelder}(i) and  Lemma \ref{lem:ergodic}(2),  we see that for all $\t\in\Theta,$ $\tilde\ell(\t)\in\R.$
Hence, using the Lemma \ref{lem:ergodic}(1) for all $ \t\in\Theta,$  
\begin{equation*}
\lim_{t\to\infty}-\frac 1{2t}A_t(\t)=\tilde\ell(\t)\quad\quad P-a.s.
\end{equation*}
Moreover, using again Assumptions \ref{ass:girsanov} and  \ref{ass:hoelder}(i)  we can see that the family
\begin{equation}\label{eq:equicont}
 \{\frac 1{t}A_t(\t)\}_{t>0} \mbox{ is equicontinuous }\ P-a.s.
 \end{equation}

 Indeed, 
\begin{equation*}\frac 1t|A_t(\t)-A_t(\t')|\leq C|\t-\t'|^{\kappa}\frac 1t\int_0^tK^2(X_s)ds,
\end{equation*}
where $C=2[Diam(\Theta)]^{\kappa}$ and $K$ given by the Assumptions \ref{ass:hoelder}(i) is sub-polynomial.
Using ergodic theorem, which holds thanks to the Lemma\ref{lem:ergodic}, $\frac 1t\int_0^tK^2(X_s)ds$ converges almost surely to some finite limit. Hence \eqref{eq:equicont} follows. As a consequence,
\begin{equation}\label{eq:uniflim}
\lim_{t\to\infty}\sup_{\t\in\Theta}\left |-\frac 1{2t}A_t(\t)-\tilde\ell(\t)\right|=0\quad\quad P-a.s.
\end{equation}
Denote \[A_t(\t,\t'):=<M_t(\t)-M_t(\t')>_t.\] 
 Using Assumptions \ref{ass:girsanov} and \ref{ass:hoelder}(i), for all $(\t,\t')\in\Theta^2,$
\[A_t(\t,\t')\leq |\t-\t'|^{2\kappa}V_t, \]
 where $V_t:=\int_0^t (\frac{K^2(X_s)}{\sigma^2(X_s)} \vee 1) ds\rightarrow\infty,\quad\mbox{if}\quad t\to\infty.$  Therefore all assumptions of the Theorem 2 in \cite {LL} are satisfied. As a conclusion, the family $\{\frac {M_t(\t)}{A_t(\t)}; \t\in\Theta, t\geq 0\}$ satisfies the Uniform Law of Large Numbers on any compact $K\in\Theta$ not containing $\ts,$ i.e. 
 \begin{equation*}
 \lim_{t\to\infty}\sup_{\t\in K}\left |\frac {M_t(\t)}{A_t(\t)}\right|=0
 \end{equation*}
We deduce, using \eqref{eq:uniflim}, that
 \begin{equation*}
 \lim_{t\to\infty}\sup_{\t\in K}\left |\frac {M_t(\t)}{t}\right|=0
 \end{equation*}
 and hence, $P-a.s.$ 
\begin{equation}\label{eq:critlimit}
\sup_{\t\in K}|t^{-1}\tilde\ell_t(\t)-\tilde\ell(\t)|\to 0.
\end{equation}
 We can now derive the a.s. consistency of $\htet$ following classical Wald's  method.
We  refer for  instance to  Theorem  5.7 in~\cite{Vaart}  for a  simple
presentation of Wald's  approach, and stress out the fact that all convergences and hence consistency holds $P$-a.s. in our setting.
 %Wald's method relies on two ingredients: the uniform convergence of the normalized criterion to the finite limit ( \eqref {eq:critlimit} here), and the following identifiability condition
%\begin{equation}\label{eq:compar}
%\sup_{\t:\ d(\t,\ts)\geq \eps} \ell(\t)<\ell(\ts)
%\end{equation}
%
Indeed, observe  that 
\begin{equation}\label {eq: inegell}
\tilde\ell(\t)\leq 0,\quad \tilde\ell(\t)=0\quad \Longleftrightarrow\quad \t=\ts
\end{equation}
and hence
\begin{equation}\label{eq:compar}
\sup_{\t:\ d(\t,\ts)\geq \eps} \tilde\ell(\t)<\tilde\ell(\ts)
\end{equation}
is trivially satisfied in our case.
We deduce from \eqref{eq:critlimit} and \eqref{eq:compar} that $P$-a.s. for all $\eps>0,$
\[ \lim_{t\to\infty}\sup_{d(\t,\ts)\geq\eps}\frac1t\tilde\ell_t(\t)<\tilde\ell(\ts)
\]
and hence for $t>t(\w)$ large enough
\[\sup_{d(\t,\ts)\geq\eps}\tilde\ell_t(\t)<\tilde\ell_t(\ts)
\]
 and finally for $t>t(\w)$,
 \[ d(\bar{\t}_t,\ts)<\eps,\]
 which means the a.s. consistency.
 \end{proof}

%%%%%%%%%%%%%%%%%%%%%%%%%%%%%%%
%% TCL%%%%%%%%%
%%%%%%%%%%%%%%%%%%%%%%%%%%%%%%%
%Denote  $\nabla_{\t}\tilde\ell(\theta)$ the gradient column vector  and $ \partial^2_{\t}\ell(\t,X):=\left(\partial_{\theta_{i},\theta_{j}}^2\ell(\theta,X)\right)_{i,j}$ the Hessian matrix of $\ell$.
Recall that $I$ is the Fisher information given by \eqref{eq:fisher}.

The next result is a central limit theorem for the estimation error. 
It is  important for us in the sequel, since the asymptotic variance serves as a benchmark for the case of discrete observations. 
% Aussi par ce que ce thm permit de démontrer la normalité as pour Fmle
\begin{thm}\label{thm:cont_clt}
Suppose that Assumptions \ref{ass:existence}--\ref{ass:fischer} hold. Then the MLE $\bar{\theta}_{t}$
is asymptotically normal:
\[
t^{1/2}(\bar{\theta}_{t}-\ts)\stackrel{\L}{\to} \NN(0,I^{-1}(\ts))\quad \mbox{as}\quad t\to\infty.
\]
%in distribution under $P$.
\end{thm}
\begin{proof}
Due to Assumptions \ref{ass:girsanov} and \ref{ass:hoelder}, Theorem 2.2 in \cite{hutton} and Theorem 1 in \cite{LL} for all $t>0$ the criterion function $\tilde\ell_t(\t, X)$ is twice continuously differentiable in $\theta.$

From (\ref{eq:ellt}) the score function can be written as $\nabla_{\t}\ell=\nabla_{\t}\tilde\ell=(\partial_{\t_1}{\tilde\ell}_{t},\ldots, \partial_{\t_d}{\tilde\ell}_{t})^T$ where
\begin{equation}\label{eq:partialel}
\partial_{\t_i}{\tilde\ell}_{t}(\theta)=-\int_{0}^{t}\frac{(b(\theta,X_s)-b(\theta^{{\star}},X_s))\partial_{\t_i}{b}(\theta,X)}{\sigma^{2}(X_s)}ds+\int_{0}^{t}\frac{\partial_{\t_i}{b}(\theta,X_s)}{\sigma(X_s)}dW_{s},
\end{equation}
for $i=1,\ldots, d.$
 A Taylor expansion around
$\bar{\theta}_t$ yields
\begin{equation}\label{eq:like_Taylor}
\int_0^1\frac1t\partial^2_{\t}\tilde\ell_t(\ts+s(\bar \t_t-\ts))ds\times\sqrt t(\bar\t_t-\ts)=- \frac1{\sqrt t}\nabla_{\t}\tilde\ell_t(\ts).
\end{equation}
 Hence,
to obtain a CLT for the estimation error $t^{1/2}(\bar{\theta}_{t}- \ts)$
we will first show the convergence of the right hand side in (\ref{eq:like_Taylor}).
The equation \eqref{eq:partialel} gives for  $\theta=\theta^{{\star}}$ 
\[
\nabla_{\t}{\tilde\ell}_{t}(\theta^{{\star}})=\int_{0}^{t}\frac{\nabla_{\t}{b}(\theta^{{\star}},X_s)}{\sigma(X_s)}dW_{s}
\]
such that the central limit theorem for multidimensional local martingales \cite{KuSo}
gives 
\begin{equation}
\label{eq:norm_ctn_score}
t^{-1/2}\nabla_{\t}{\tilde\ell}_{t}(\theta^{{\star}})=t^{-1/2}\int_{0}^{t}\frac{\nabla_{\t}{b}(\theta^{{\star}},X)}{\sigma(X_s)}dW_{s}\stackrel{\L}{\to} \NN(0,I).
\end{equation}
In the next step we prove the convergence of 
\begin{equation*} %\label{eq:intlpp}
\int_0^1\frac1t \partial^2_{\t}\tilde\ell_t(\ts+s(\bar \t_t-\ts))ds.
\end{equation*}
From \eqref{eq:partialel} we see that for $(i,j) \in \{1,\dots,d\}$,
\begin{align} 
\nonumber
\partial^2_{\t_i\t_j}{\tilde\ell}_{t}(\theta) & =-\int_{0}^{t}\frac{(b(\theta,X_s)-b(\theta^{{\star}},X_s))\partial^2_{\t_i,\t_j}{b}(\theta,X_s)}{\sigma^{2}(X_s)}ds-\int_{0}^{t}\frac{\partial_{\t_i}{b}(\theta,X_s)\partial_{\t_j}{b}(\theta,X_s)}{\sigma^2(X_s)}ds
\\ \quad \nonumber
 &+\int_{0}^{t}\frac{\partial^2_{\t_i\t_j}{b}(\theta,X_s)}{\sigma(X_s)}dW_{s}
 \\ & 
 \label{eq:intpp}
 := U_{t}^{1}(\t)+U_{t}^{2}(\t)+U_{t}^{3}(\t).
\end{align}
Using the ergodic theorem, $P$-a.s.
\[ \frac 1t U_t^1(\t)\to
U^1_\infty(\theta):=-
\int_{\R} \frac{(b(\theta,x)-b(\theta^{{\star}},x))\partial^2_{\t_i\t_j}b(\theta,x)}{\sigma^{2}(x)}\pi(dx)
;\]
\[\frac 1t U_t^2(\t)\to
U^2_\infty(\theta):=-
\int_{\R}\frac{ \partial_{\t_i}{b}(\theta,x)\partial_{\t_j}{b}(\theta,x)              }{\sigma^2(x)}\pi(dx)
=-I_{i,j}(\t).\]
Moreover, using 
 Assumption \ref{ass:hoelder} and \ref{ass:subpolynom} and the same argument which were used to prove the equicontinuity \eqref{eq:equicont} we obtain that
%these convergences are uniform with respect to $\theta\in\Theta$ : indeed,  \[
% \frac{ | U_t^l(\t)- U_t^l(\t')|}t\leq|\t-\t'|^{\alpha}\frac1t\int_0^tC(X_s)ds,
% \]
%where $0<\alpha<1,$  and $C:\R\to [0,\infty)$ is sub-polynomial. Using ergodic theorem, which holds thanks to the Lemma \ref{lem:ergodic},  $\frac1t\int_0^tC(X_s)ds$ converges almost surely to some finite. Hence,
 the families of functions
$( \t \mapsto \frac{U^1_t}{t}(\t))_{t \ge 0}$ and  $( \theta \mapsto \frac{U^2_t}{t}(\t))_{t \ge 0}$ are almost surely equicontinuous. 
Finally, the uniform  law of large numbers  for local martingales \cite{LL} together with Assumptions \ref{ass:girsanov} ,\ref{ass:hoelder} and  \ref{ass:subpolynom}  gives that $P$-a.s.
\[
\sup_{\t\in\Theta}t^{-1}|U_{t}^{3} (\t)|=\sup_{\t\in\Theta}t^{-1}|\int_{0}^{t}\frac{\partial^2_{\t}{b}(\theta,X_s)}{\sigma(X_s)}dW_{s}|
{\to}0
\]

Using \eqref{eq:intpp} and the four last displays we obtain $P$-a.s.
\begin{equation}
\sup_{\theta\in\Theta}\left|t^{-1}\partial^2_{\t}{\tilde\ell}_{t}(\theta)-(U^1_{\infty}(\t)-I(\t))\right|{\to}0\label{eq:lddot2}
\end{equation}
 Using this uniformity together with a.s. convergence $\bar{\theta}_{t}\to\theta^{{\star}}$  we get $P$-a.s.

\begin{equation*}
\sup_{s \in [0,1]}\left|t^{-1}\partial^2_{\t}{\tilde\ell}_{t}(\ts+s(\bar \t_t-\ts))-(-I(\ts))\right|{\to}0
\end{equation*}

%
%This implies that, almost surely, the difference 
%\[\int_0^1 \frac{1}{t} U_t^l(\ts+s(\bar \t_t-\ts))ds - \int_0^1
%U_\infty^l(\ts+s(\bar \t_t-\ts)) \dd s
%\] converges to $0$, as $t \to \infty$, for $l=1,2$.  
%
%From the consistency of the estimator $\bar{\theta}_t$, we deduce that almost surely,
%
%\begin{equation}\label{eq:intppU1}
%\int_0^1 t^{-1}U_{t}^{1}(\ts+s(\bar \t_t-\ts))ds\to  U^1_\infty(\theta^\star) = 0
%\end{equation}
%
%and
%\begin{equation}
%\label{eq:intppU2}
% \int_0^1 t^{-1} U_t^2(\ts+s(\bar \t_t-\ts) ds\to
% U_\infty^2(\ts)  = -I_{i,j}(\ts).
%\end{equation}
%\begin{equation}
%\label{eq:intppU3}
%\int_0^1 t^{-1}U_{t}^{3} (\ts+s(\bar \t_t-\ts)) ds
%{\to}0
%\end{equation}
%Collecting together \eqref{eq:intpp} to \eqref{eq:intppU3}, we deduce,
%%Finally, we have 
%that $P$-a.s.
and
\begin{equation}
\label{eq:cv_intlpp}
t^{-1} \int_0^1\partial^2_{\t}{\tilde\ell}_{t}(\ts+s(\bar \t_t-\ts))ds
%=t^{-1}(U_{t}^{1}+U_{t}^{2}+U_{t}^{3})(\tilde{\theta}_{t})
{\to}-I(\ts).
\end{equation}

Finally, from the non-degeneracy of the Fisher information matrix $I(\ts)$, \eqref{eq:norm_ctn_score}, \eqref{eq:cv_intlpp}, and Slutsky's theorem, we deduce the asymptotic normality of the estimator.
\end{proof}
%%%%%%%%%%%%%%%%%%%%%%%%%%%%%%%
\subsection{Local asymptotic normality and efficiency}
%%%%%%%%%%%%%%%%%%%%%%%%
To obtain an asymptotic efficiency result in the sense of Hàjek-Le Cam's convolution theorem we prove now the local asymptotic normality property for the  
statistical experiment $(\Omega,\mathcal{F},(\mathcal{F}_{t}),\mathcal{P})$.
 From this result we can then deduce later on efficiency of the discretized estimator with jump filter (cf. Theorem \ref{thm:anfinact} and \ref{thm:aninfinact}).
\begin{thm}\label{thm:LAN}
Suppose that Assumptions \ref{ass:existence} to \ref{ass:fischer} are satisfied. Then 
 the family $(P^{\t})_{\t\in\Theta}$ is locally asymptotically normal. That is, for all $h \in \mathbb{R}^d$, we have the convergence in distribution under $P$,
\begin{equation}\label{al:LAN}
%&\lim_{t\to\infty}(\ell_t(\ts+\frac h{\sqrt t})-\ell_t(\ts))=....
\ell_t(\ts+\frac h{\sqrt t})-\ell_t(\ts) \stackrel{\L}{\to} -1/2 h^{\top} I(\ts) h + N,\quad \mbox{as}\quad t\to\infty,
\end{equation} 
where $N \sim \mathcal{N}(0, h^{\top} I(\ts) h)$.
%in distribution.
%under $P^\star$ distribution.
% and $\bar\t$ is efficient.
 %\xrightarrow{t \to \infty}%
 As a consequence the drift estimator $\bar{\theta}_{t}$ is asymptotically efficient
in the sense of the Hájek-Le Cam convolution theorem.

\end{thm}
\begin{proof}
%\begin{align{\star}}
%l_{t_{n}}(\theta,X)-l_{t_{n}}^{n}(\theta,X)) & =\int_{0}^{t_{n}}\sigma(X)^{-2}b(\theta,X)\; dX_{s}^{c}-\sum_{i=1}^{n}\sigma(X)^{-2}b(\theta,X)b_{i}X\mathbf{1}_{|b_{i}X|\leq v_n}\\
% & +\frac{1}{2}\int_{0}^{t_{n}}\sigma(X_{s-})^{-2}b(\theta,s,X_{s})^{2}\; ds-\frac{1}{2}\sum_{i=1}^{n}\sigma(X)^{-2}b(\theta,X)^{2}b_{i}\\
% & =A_{n}^{1}+A_{n}^{2}.
%\end{align{\star}}
%
\begin{align*}
&\ell_t(\ts+\frac h{\sqrt t})-\ell_t(\ts)=-\frac 12\int_0^t\frac {(b(\ts+\frac h{\sqrt t},\Xs)-b(\ts, \Xs))^2ds}{\sigma^2(X_s)} \\ &+\int_0^t\frac {(b(\ts+\frac h{\sqrt t},\Xs)-b(\ts, \Xs))}{\sigma(\Xs)}dW_s\\
&=-\frac {1}{2}\int_0^1\int_0^1\left (  \frac 1t\int_0^t
\frac{h^{\top}(\nabla b(\ts+\frac{hu}{\sqrt t},\Xs)\nabla b^{\top}(\ts+\frac{hu'}{\sqrt t},\Xs)h}{\sigma^2(\Xs)}ds\right)dudu' \\
& + \frac 1{\sqrt t}\int_0^t\frac{\nabla b^T(\ts,\Xs)h}{\sigma(\Xs)}dW_s+R_t.
\end{align*}
Where
\[R_t:=\int_0^t\frac {(b(\ts+\frac h{\sqrt t},\Xs)-b(\ts, \Xs))}{\sigma(\Xs)}dW_s-\frac 1{\sqrt t}\int_0^t\frac{\nabla b^T(\ts,\Xs)h}{\sigma(\Xs)}dW_s.\]
Using  Assumption \ref{ass:hoelder} and the ergodic theorem,
for all fixed $r>0, r'>0$ such that $\ts+r\in\Theta$, $\ts+r'\in\Theta$  we obtain 
\[\lim_{t\to\infty} \frac 1t\int_0^t
\frac{h^{\top}\nabla b(\ts+r,\Xs)\nabla b^{\top}(\ts+r',\Xs)h}{\sigma^2(\Xs)}ds=\int_{\R}\frac{h^{\top}\nabla b(\ts+r,x)\nabla b^{\top}(\ts+r',x)h}{\sigma^2(\Xs)}d\pi(x)\] 
 $P$-a.s. and Assumption \eqref{ass:subpolynom} and Lemma \ref{lem:ergodic} imply that this last limit is finite. Moreover, using Assumption \ref{ass:hoelder} it can be shown that this convergence is uniform, hence for $hu/\sqrt t \to 0$ it gives that $P-a.s.$
\begin{multline} \label{eq:info_lan}
\lim_{t\to\infty}\int_0^1du\int_0^1du' \frac {1}{t}\int_0^t
\frac{h^{\top}\nabla b(\ts+\frac{hu}{\sqrt t},\Xs)\nabla b^{\top}(\ts+\frac{hu'}{\sqrt t},\Xs)h}{\sigma^2(\Xs)}
ds\\=\int_{\R}\frac{h^{\top}\nabla b(\ts,x)\nabla b^{\top}(\ts,x)h}{\sigma^2(\Xs)}d\pi(x)=h^{\top}I(\ts)h.
\end{multline} 

Using Markov inequality
\begin{equation}\label{eq:remainder_lan}
P(|R_t|\geq \eps)\leq \frac{Var R_t}{\eps^2}\leq \frac {\|h\|^2}{\eps^2}\frac 1t\int_0^t\left(\frac {\|h\|}{\sqrt t}\right)^{2\kappa}E\left(\frac {K_1^2(\Xs)}{\sigma^2(\Xs)}\right)ds,
\end{equation}
where $K_1$ is a Holder constant of $\nabla b$ is supposed to be at most of polynomial growth. Using ergodic theorem in mean, we obtain  $R_t\to 0$ in $P$ probability.

Due to the CLT for martingales in \cite{KuSo} 
\[
\frac 1{\sqrt t}\int_0^t\frac{\nabla b^{\top}(\ts,\Xs)h}{\sigma(\Xs)}dW_s\to { \NN}(0, h^{\top}I(\ts)h)
\]
in distribution.
Combining the latter equation with \eqref{eq:info_lan}--\eqref{eq:remainder_lan}, we obtain \eqref{al:LAN}. This implies together with Theorem \ref{thm:cont_clt} that $\bar{\t}_t$ is asymptotically efficient in the 
sense of the Hájek-Le Cam convolution theorem.
\end{proof}

\subsection{Proofs of Theorems \ref{thm:tnconsitency}, \ref{thm:anfinact} and \ref{thm:aninfinact}.}

\begin{proof}[Proof of Theorem \ref{thm:tnconsitency}]
Let $\tilde\ell:\Theta\to \R$ be given by \eqref{eq:ell} and 
define\begin{equation}\label{def:ell}
\ell(\t)=\tilde \ell(\t)+\frac 12\pi\left(\frac{b^2(\ts,x)}{\sigma^2(x)}\right).
\end{equation}
Under Assumptions \ref{ass:existence} and \ref{ass:girsanov}  the last term in the right hand side of \eqref{def:ell} is finite.%Given \eqref{eq:compar} 
We will apply Wald's method for proving consistency of $M$ estimators (see for example Theorem 5.7 in \cite{Vaart}). It follows from \eqref {eq: inegell} that
\begin{equation}
\sup_{\t;\  d(\t,\ts)\geq \eps}\ell(\t)\leq \ell(\ts). 
\end{equation}
 Therefore, it remains to prove that 

\begin{equation*}
\lim_{n\to\infty}\sup_{\t\in\Theta}|t_n^{-1}\ell^n_{t_n}(\t)-\ell(\t)|= 0\quad \text{in probability}.
\end{equation*}
To obtain this last statement we decompose this difference as follows:
\begin{equation}\label{eq:condecomp}
\sup_{\theta\in\Theta}|\ell(\theta)-t_{n}^{-1}\ell_{t_{n}}^{n}(\theta)|\leq\sup_{\theta\in\Theta}|\ell(\theta)-t_{n}^{-1}\ell_{t_{n}}(\theta)|+\sup_{\theta\in\Theta}|t_{n}^{-1}(\ell_{t_{n}}(\theta)-\ell_{t_{n}}^{n}(\theta))|.
\end{equation}

 Using respectively the Ergodic Theorem given by Lemma \ref{lem:ergodic} (1) and the Law of Large Numbers for continuous local martingales (\cite{RY} p.178) we see that a.s.
\[\frac1t\int_0^t\frac{b^2(\ts,X_s)}{\sigma^2(X_s)}ds\to \pi\left(\frac{b^2(\ts,x)}{\sigma^2(x)}\right)\] 
and 
\[\frac 1t\int_0^t\frac{b(\ts,X_s)}{\sigma(X_s)}dW_s\to 0\]Using these two last display and \eqref{eq:critlimit} we see that the first term of the decomposition \eqref{eq:condecomp}
 tends to zero $P$-a.s..
In order to show the  convergence to zero in probability of the second term, we decompose it as follows.
\begin{align*}
 & \sup_{\t\in\Theta}\left |t_{n}^{-1}(\ell_{t_{n}}(\theta)-\ell_{t_{n}}^{n}(\theta))\right |\\
 & \leq \sup_{\t\in\Theta}t_{n}^{-1}\left |\int_{0}^{t_{n}}\sigma(X_{s})^{-2}b(\theta,X_{s})\; dX_{s}^{c}-\sum_{i=1}^{n}\sigma(X_{t_{i-1}})^{-2}b(\theta,X_{t_{i-1}})\Delta_{i}^nX\mathbf{1}_{|\Delta_{i}^nX|\leq v_n}\right |\\
 & +\sup_{\t\in\Theta}t_{n}^{-1}\left |\frac{1}{2}\int_{0}^{t_{n}}\sigma(X_{s})^{-2}b(\theta,X_{s})^{2}\; ds-\frac{1}{2}\sum_{i=1}^{n}\sigma(X_{t_{i-1}})^{-2}b(\theta,X_{t_{i-1}})^{2}\Delta_{i}^nId\right |\\
 & =\sup_{\t\in\Theta}t_{n}^{-1}|A_{n}^{1}(\t)|+\sup_{\t\in\Theta}t_{n}^{-1}|A_{n}^{2}(\t)|.
\end{align*}
Hence, it remains to prove the convergence to zero of $t_n^{-1}|A_{n}^{1}(\t)|$ and $t_n^{-1}|A_{n}^{2}(\t)|$ uniformly in $\theta$. 
For $t_n^{-1}|A_{n}^{1}(\t)|$ we apply Proposition \ref{prop:jumpfiltering} in the finite activity case and Proposition \ref{prop:jumpfiltering_infinit}  in the case of infinite activity, together with the fact that $n \Delta_n=O(t_n)$. Indeed, using Assumption \ref{ass:hoelder} and \ref{ass:subpolynom} we see that the function $f(\t,x)=\sigma(x)^{-2}b(\theta,x)^2$
satisfies all assumptions of Propositions \ref{prop:jumpfiltering} or \ref{prop:jumpfiltering_infinit}.
% namely, 
%\begin{itemize}
%\item for all $x\in\R,$ $f$ is Hölder continuous in $\t$,
%with sub-polynomial Hölder constant $C(x)$; 
%\item for all $\t\in\Theta,$ $f\in\C^2(\R)$ and $\sup_{\t\in\Theta}|f(\t,x)|$, %$\sup_{\t\in\Theta}|f'_x(\t,x)|$ and $\sup_{\t\in\Theta} |f''_x(\t,x)|$ are sub-polynomial. 
%\end{itemize}   
For the second term $t_n^{-1}|A_{n}^{2}(\t)|$ we use Lemma \ref{lem:Riemann_app}.
\end{proof}
%%%%%%%%%%%%%%%%%%%%%%%%%%%%%%%
\begin{proof}[Proof of Theorem \ref{thm:anfinact}]
 A Taylor expansion around
$\hat{\theta}_n$ yields
\begin{equation}\label{eq:like_ATaylor}
\frac1{t_n}\int_0^1\partial^2_{\t}\ell^n_{t_n}(\ts+s(\hat \t_n-\ts))ds\times t_n^{1/2}(\hat\t_n-\ts)=- \frac1{ t_n^{1/2}}\nabla_{\t}\ell^n_{t_n}(\ts).
\end{equation}
%1. A Taylor expansion in $\ts$ yields
%\begin{equation}
%t_{n}^{1/2}(\hat{\theta}_n-\ts)=\frac{t_{n}^{-1/2}\dot{\ell}_{t_{n}}^{n}(\ts)}{t_{n}^{-1}\ddot{\ell}_{t_{n}}^{n}(\tilde{\theta}_n)},\label{eq:taylor_exp}
%\end{equation}
 For the right hand side  we find that
\begin{equation}
\frac1{{ t_n}^{1/2}}\nabla_{\t}\ell^n_{t_n}(\ts)=\frac{\nabla_{\t}{\ell}_{t_{n}}^{n}(\ts)-\nabla_{\t}{\ell}_{t_{n}}(\ts)}{\sqrt{t_{n}}}+\frac{\nabla_{\t}{\ell}_{t_{n}}(\ts)}{\sqrt{t_{n}}}.\label{eq:l_dot}
\end{equation}
By  \eqref{eq:norm_ctn_score} we have that under $P$
\begin{equation}
\frac{\nabla_{\t}{\ell}_{t_{n}}(\ts)}{\sqrt{t_{n}}}\stackrel{\L}{\to}N(0,I(\ts)),\quad n\to\infty.\label{eq:ldot_norm}
\end{equation}
The first term of the sum on the right hand side of (\ref{eq:l_dot}) has the
form
\begin{align*}
 & \frac{\nabla_{\t}{\ell}_{t_{n}}^{n}(\ts)-\nabla_{\t}{\ell}_{t_{n}}(\ts)}{t_{n}^{1/2}}\\
 & =-t_{n}^{-1/2}\left(\int_{0}^{t_{n}}\sigma(X_{s})^{-2}\nabla_{\t}{b}(\ts,X_{s})\; dX_{s}^{c}-\sum_{i=1}^{n}\sigma(X_{t_{i-1}})^{-2}\nabla_{\t}{b}(\ts,X_{t_{i-1}})\Delta_{i}^nX\mathbf{1}_{|\Delta_{i}^nX|\leq v_n}\right)\\
 & +t_{n}^{-1/2}\frac{1}{2}\left(\int_{0}^{t_{n}}\sigma(X_{s})^{-2}\nabla_{\t}{b}(\ts,X_{s})^{2}\; ds-\sum_{i=1}^{n}\sigma(X_{t_{i-1}})^{-2}\nabla_{\t}{b}(\ts,X_{t_{i-1}})^{2}\Delta_{i}^nId\right).
\end{align*}
By applying Proposition \ref{prop:jumpfiltering} for $k=1,\ldots d$ with $f_k(\ts,x)=\sigma(x)^{-2}\partial_{\t_k}{b}(\ts,x)$, and using Assumptions
\ref{ass:hoelder}-- \ref{ass:subpolynom}
we obtain that
%%%%%%%%ICI APPLIQUER la version non-uniforma de la proposition 16

\[t_{n}^{-1/2}\left(\int_{0}^{t_{n}}\sigma(X_{s})^{-2}\partial_{\t_k}{b}(\ts,X_{s})\; dX_{s}^{c}-\sum_{i=1}^{n}\sigma(X_{t_{i-1}})^{-2}\partial_{\t_k}{b}(\ts,X_{t_{i-1}})\Delta_{i}^nX\mathbf{1}_{|\Delta_{i}^nX|\leq v_n}\right)\stackrel{P}{\longrightarrow} 0
\]
as $n\to\infty$.  Furthermore, Lemma \ref{lem:Riemann_app} (ii) leads
to 

\begin{equation*}
t_{n}^{-1/2}\left(\int_{0}^{t_{n}}\sigma(X_{s})^{-2}\partial_{\t_k}{b}(\ts,X_{s})^{2}\; ds-\sum_{i=1}^{n}\sigma(X_{t_{i-1}})^{-2}\partial_{\t_k}{b}(\ts,X_{t_{i-1}})^{2}\Delta_{i}^nId\right)
\stackrel{P}{\longrightarrow} 0,
%O_{P}\left(n^{1/2}\Delta_{n}^{3/2}\right). 
\end{equation*}
as $n\to\infty.$
Combining now the last three displays results in
\[
\frac{\nabla_{\t}{\ell}_{t_{n}}^{n}(\ts)-\nabla_{\t}{\ell}_{t_{n}}(\ts)}{t_{n}^{1/2}}\stackrel{P}{\longrightarrow}0
\]
such that (\ref{eq:l_dot}) and (\ref{eq:ldot_norm}) give
\[
t_{n}^{1/2}\nabla_{\t}{\ell}_{t_{n}}^{n}(\ts)\stackrel{d}{\to}N(0,I(\ts)),\quad n\to\infty.
\]
To finish the proof it remains to show the convergence of the left hand side in \eqref {eq:like_ATaylor}.
For $(j,k)\in\{1,\ldots,d\}^2$ and $\t\in\Theta,$ 
\begin{align*}
 & t_{n}^{-1}\sup_{\theta\in\Theta}\left|\left(\partial^2_{\t_j\t_k}{\ell}_{t_{n}}^{n}(\theta)-\partial^2_{\t_j\t_k}{\ell}_{t_{n}}(\theta)\right)\right|\\
 & \leq t_{n}^{-1}\sup_{\theta\in\Theta}\left|\int_{0}^{t_{n}}\sigma(X_{s})^{-2}\partial^{ 2}_{\t_j\t_k}{b}(\theta,X_{s})\; dX_{s}^{c}-\sum_{i=1}^{n}\sigma(X_{t_{i-1}})^{-2}\partial_{\t_j\t_k}^{ 2}{b}(\theta,X_{t_{i-1}})\Delta_{i}^nX\mathbf{1}_{|\Delta_{i}^nX|\leq v_n}\right|\\
 & +t_{n}^{-1}\sup_{\theta\in\Theta}\left|\int_{0}^{t_{n}}{{\sigma(X_s)^{-2}}\partial_{\t_j}{b}(\theta,X_s)\partial_{\t_k}{b}(\theta,X_s)}\; ds-\sum_{i=1}^{n}\sigma(X_{t_{i-1}})^{-2}{\partial_{\t_j}{b}(\theta,X_s)\partial_{\t_k}{b}(\theta,X_s)}\Delta_{i}^nId\right|\\
 &+
  t_{n}^{-1}\sup_{\theta\in\Theta}\left|\int_{0}^{t_{n}}{{\sigma(X_s)^{-2}}\partial^2_{\t_j\t_k}{b}(\theta,X_s){b}(\theta,X_s)}\; ds-\sum_{i=1}^{n}\sigma(X_{t_{i-1}})^{-2}{\partial^2_{\t_j\t_k}{b}(\theta,X_s){b}(\theta,X_s)}\Delta_{i}^nId\right|
\\
 & =U_{n}^{1}+U_{n}^{2}+{U_n^3}.
\end{align*}
Proposition \ref{prop:jumpfiltering} together with Assumptions \ref{ass:hoelder}--  \ref{ass:subpolynom} state that %\todo{Why do the assumptions of Prop 5 hold?}
\begin{equation}
U_{n}^{1}\stackrel{P}{\to}0,\quad\text{as \ensuremath{n\to\infty}. }\label{eq:lddot3}
\end{equation}
Lemma \ref{lem:Riemann_app} (i) gives {for k=2,3}
\begin{equation}
U_{n}^{k}\stackrel{P}{\to}0,\quad\text{as \ensuremath{n\to\infty}. }\label{eq:lddot4}
\end{equation}
 Combining \eqref{eq:lddot3} and \eqref {eq:lddot4} with consistency of $\hat \t$ and $\bar\t$ we  get
\begin{equation*} %\label{eq:intlpp}
\int_0^1\frac1{t_n} |\partial^2_{\t}\ell^n_{t_n}(\ts+s(\hat \t_n-\ts))-\partial^2_{\t}\ell_{t_n}(\ts+s(\bar \t_{t_n}-\ts))|ds
\stackrel{P}{\to}0,\end{equation*}
and hence, using \eqref{eq:cv_intlpp}
\begin{equation*} %\label{eq:intlpp}
\frac1{t_n}\int_0^1 \partial^2_{\t}\ell^n_{t_n}(\ts+s(\hat \t_n-\ts))ds \stackrel{P}{\to} -I(\ts)
\end{equation*}
%
%\[
%\sup_{\theta\in\Theta}\left|t_{n}^{-1}\ddot{l}_{t_{n}}^{n}(\theta)-I\right|\stackrel{P}{\to}0
%\]
as $n\to\infty$ such that the result follows.
\end{proof}
%%%%%%%%%%%%%%%%%%%%%%%%%%%%%%%%%%%%%%%%%%%%%%
%AsNorm infinite activity
\begin{proof}[Proof of Theorem \ref{thm:aninfinact}]By replacing in the previous proof Proposition \ref{prop:jumpfiltering} by Proposition \ref {prop:jumpfiltering_infinit}
we obtain the result for the infinite activity case.
\end{proof}

\section{Proofs for jump filtering}\label{sec:proof_filtering}
In this section we prove the results that were used in the Section \ref{sec:jumpfilter} to obtain the convergence of the jump filter (cf. Proposition \ref{prop:jumpfiltering} and \ref{prop:jumpfiltering_infinit}) to integral functionals with respect to the continuous martingale part of $X$. We start by proving the Lemma \ref{lem:filter_error_finite} that shows the convergence of the jump filter approximation to the continuous part in the finite activity case.

 We recall some notations: $\mu$ denotes the Poisson random  measure
on $[0,\infty)\times \mathbb{R}$ associated with the jumps of the L\'evy process $L$, 
the intensity of this jump measure is $ ds \times \nu(dz)$. We define
$\tilde{\mu}=\mu-ds \times \nu(dz)$ as the compensated Poisson measure such that we have $L_t=\int_0^t \int_\mathbb{R} z \mu(ds,dz)$.
In the specific situation where the Lévy process $L$ has a finite intensity $\nu(\mathbb{R})<\infty$, 
% it is a compound Poisson process and 
we shall denote by $N_t=\int_0^t \int_\mathbb{R}  \mu(ds,dz)$ the process that counts the number of jumps up to time $t$.

\begin{proof}[Proof of Lemma \ref{lem:filter_error_finite}]
	
	For all $n\in\N^*,\ i\in\N^*$ we define the set where increments of $X$ are small:
	\begin{equation}\label{eq:kni}
	K_{n}^{i}=\left\{ |\Delta_{i}^nX|\leq v_n\right\},
	\end{equation}
	the event that $L$ and so also $X$ do not jump:
	\begin{equation}\label{eq:mni}
	M_{n}^{i}=\left\{ \Delta_i^n N=0\right\},
	\end{equation}
	and the event that an increment of the jump part is small:
	\begin{equation}\label{eq:dni}
	D_n^{i}=\left\{\left |\Delta_i^nX^J\right |\leq \frac {v_n}3\right\},
	\end{equation}
	where we denoted by $X^J$  the jump part of $X$ given by 
	\[
	X^J_t=\int_{0}^{t}\int_{\R\setminus\{0\} }\gamma(X_{s-})  z\mu(ds,dz),\quad t\geq 0.
	\]
	We start by proving  (i).
	Using the previously defined sets we introduce the following quantities.
	\begin{align}\label{al:G123}
	& G^{1}_n(\t):=\sum_{i=1}^{n}f(\theta,X_{t_{i-1}})\left(\Delta_i^nX^J\right)1_{K_{n}^{i}\cap\left(M_{n}^{i}\right)^{c}},\\ 
	\label{E:Gn2}
	& G^{2}_n(\t):=\sum_{i=1}^{n}f(\theta,X_{t_{i-1}})\left(\Delta_{i}^nX^{c}\right)1_{\left(K_{n}^{i}\right)^{c}\cap D_{n}^{i}},\\
	\label{E:Gn3}
	&G^{3}_n(\t):=\sum_{i=1}^{n}f(\theta,X_{t_{i-1}})\left(\Delta_{i}^nX^{c}\right)1_{\left(K_{n}^{i}\right)^{c}\cap \left(D_{n}^{i}\right){}^{c}}, 	
	\end{align}
	%& =G^{1}_n(\t)+G^{2}_n(\t)+G^{3}_n(\t).
	and decompose the  difference to be estimated as follows:
	\begin{equation}\label{al:decoforun}
	\sum_{i=1}^{n}f(\theta,X_{t_{i-1}})\left(\Delta_{i}^nX^{c}-\Delta_{i}^nX\mathbf{1}_{|\Delta_{i}^nX|\leq v_n}\right)
	=G^{1}_n(\t)+G^{2}_n(\t)+G^{3}_n(\t).
	\end{equation}
	%Hence, it remains to find the uniform in $\t$ bound for
	% these  three terms. 
	% 
	To prove the convergence of $G_n^1(\t)$ we  
	decompose the set $K_n^{i}\cap(M_{n}^{i})^{c}$ into three disjoint events
	
	\begin{align}\label{eq:indicator}
	\mathbf{1}_{K_n^{i}\cap(M_{n}^{i})^{c}} &
	=\mathbf{1}_{\{\Delta^n_i N\geq2\}\cap K_n^i}+\mathbf{1}_{\{\Delta_{i}^nN=1,|\Delta_{i}^nL|\geq 2v_{n}/\gamma_{min}\}\cap K_n^i}\nonumber \\
	& +\mathbf{1}_{\{\Delta_{i}^nN=1,|\Delta_{i}^nL|<2v_{n}/\gamma_{min}\}\cap K_n^i}
	\end{align}
	
	%\begin{equation}
	%\label{eq:indicator}
	%1_{K_{n}^{i}\cap\left(M_{n}^{i}\right)^{c}}=1_{\{\Delta_{i}^nN=1,|\Delta_{i}^nL|\geq 2v_n/\gamma_{min}\}\cap K_{n}^{i}}+1_{\{\Delta_{i}^nN=1,|\Delta_{i}^nL|<2v_n/\gamma_{min}\}\cap K_{n}^{i}}+1_{\{\Delta_{i}^nN>1\}\cap K_{n}^{i}}
	%\end{equation}
	Using  Lemma \ref{lem:jac-prot} (3), the definition of $v_n$ and Markov's inequality we can see that the second indicator of this decomposition is on an event that has small probability. Indeed, for all $p>1,$
	\begin{equation*} %\label{eq:indi1}
	P\left(\{\Delta_{i}^nN=1,|\Delta_{i}^nL|\geq 2v_n/\gamma_{min}\}\cap K_{n}^{i}\right)\leq P\left(|\Delta_{i}^nX^{c}|\geq v_n\right)=O(\Delta_{n}^{p/2}v_n^{-p})=O(\Delta_{n}^{\eps p}).
	\end{equation*}
	 Then, using the $L^2$--isometry for stochastic integral with respect to the compensated 
	Poisson measure and the  Jensen's inequality, we get
	\begin{align}
	\label{al:jumpintest}
	& \E \left|\Delta_i^nX^J\right|^2\leq 2\E \left|\int_{t_{i-1}}^{t_i}\int_{\R\setminus \{ 0\}}\gamma(X_{s-})z\tilde{\mu}(ds,dz)\right|^2
	+2\E \left[\int_{t_{i-1}}^{t_i}\int_{\R\setminus \{ 0\}}\gamma(X_{s})zds\nu(dz)\right]^2 
	\nonumber
	\\
	&\leq 2 \int_{t_{i-1}}^{t_i}\int_{\R\setminus \{ 0\}}\E[\gamma^2(X_{s})]z^2ds\nu(dz)
	+ 2 \int_{t_{i-1}}^{t_i}\int_{\R\setminus \{ 0\}}\E[\gamma^2(X_{s})]|z| ds\nu(dz)
	\int_{t_{i-1}}^{t_i}\int_{\R\setminus \{ 0\}}|z| ds\nu(dz)
	\nonumber
	\\
	&
	%\leq C_x(1+\Delta_n) \int_{t_{i-1}}^{t_i}\int_{\R\setminus \{ 0\}}\E[\gamma^2(X_{s})]z^2ds\nu(dz)
	=O(\Delta_n),
	\end{align}
	where in the last line we have used 
 Assumption \ref{ass:existence}, 
	Assumption \ref{ass:ergodic} (i), Assumption \ref{ass:jumps} (ii) and
	Lemma \ref{lem:ergodic} statement $(3)$.
		 Using Hölder's inequality twice and  Lemma \ref{lem:ergodic} statement $(3)$  we get for all $p>0,$ 
	\begin{equation}
	E\sup_{\t\in\Theta}\left|\sum_{i=1}^{n}f(\theta,X_{t_{i-1}})\Delta_i^nX^J1_{\{\Delta_{i}^nN=1,|\Delta_{i}^nL|\geq 2v_n/\gamma_{\min}\}\cap K_{n}^{i}}\right|
	=O(n\Delta_{n}^{\eps p}). \label{eq:indi1}
	\end{equation}
	For the  third  indicator function in (\ref{eq:indicator}) we observe
	that
	\begin{align}\label{eq:indi2}
	& \E\sup_{\t\in\Theta}|\sum_{i=1}^{n}f(\theta,X_{t_{i-1}})\left(\Delta_i^nX^J\right)1_{\{\Delta_{i}^nN=1,|\Delta_{i}^nL|<2v_n/\gamma_{\text{min}}\}\cap K_{n}^{i}}\nonumber\\
	%& =-\sum_{i=1}^{n}f(\theta,X_{t_{i-1}})\gamma(X_{\tau_{i}-})Z_{i}1_{\{\Delta_{i}N=1,|\Delta_{i}L|<cv_n\}\cap K_{n}^{i}}=:V_{n},
	&\leq\sum_{i=1}^{n}\int_{t_{i-1}}^{t_{i}}\int_{|z|  <  2v_n/\gamma_{\text{min}}}E\left[\sup_{\t\in\Theta}|f(\t,X_{t_{i-1}})\gamma(X_{s})|\right]|z|ds\nu(dz)\nonumber \\
	& =O(n\Delta_{n}\int_{|z|  <  2v_n/\gamma_{\text{min}}}|z|\nu(dz)),
	\end{align}
	where we have used the sub-polynomial growth of  $\gamma$,  $f$ and Lemma \ref{lem:ergodic} statement(3). 
	
	For the  first  indicator in (\ref{eq:indicator}) we obtain by Hölder's inequality
	with conjugated exponents, 
	$p$, $q$, such that $p^{-1}+q^{-1}=1$, and $q^{-1}=1-\eps/2$,
	
	\begin{align}
	& E\sup_{\t\in\Theta}\left|\sum_{i=1}^{n}f(\theta,X_{t_{i-1}})
	\left(\Delta_{i}^nX^{c}-\Delta_{i}^nX\mathbf{1}_{|\Delta_{i}^nX|\leq v_n}\right)1_{\{\Delta_{i}^n N \bac \ge  2 \}\cap K_{n}^{i}}\right|\nonumber \\
	& \leq\sum_{i=1}^{n}E\sup_{\t\in\Theta}|f(\t,X_{t_{i-1}})|\left(|\Delta_{i}^nX^{c}|+v_n\right)1_{\{\Delta_{i}^nN \bac \ge  2\}}\nonumber
	\\
	& 
	\leq\sum_{i=1}^{n} 
	\left( E \sup_{\t\in\Theta} |f(\t,X_{t_{i-1}})|^p 
	(|\Delta_{i}^nX^{c}|+v_n)^p \right)^{1/p} 
	P(\Delta_{i}^n N \ge  2 )^{1/q}
	\nonumber \\
	& =O(n(\Delta_{n}^{1/2}+v_n)\Delta_{n}^{2/q}=O(n\Delta_{n}^{5/2-2\eps}),
	\label{eq:indi3}
	\end{align}
	 where we have used that $P(\Delta_{i}^n N \ge  2)=O(\Delta_n^2)$.
	
	From
	\eqref{eq:indi1}, \eqref{eq:indi2} and \eqref{eq:indi3} it follows
	that
	
	\begin{equation}\label{eq:gn1}
	E\sup_{\t\in\Theta}|G^{1}_n(\t)|\leq  O(n\Delta_{n}\int_{|z|\leq 2v_n/\gamma_{min}}|z|\nu(dz))+ O(n\Delta_{n}^{5/2-2\eps}).
	\end{equation}
	
	To estimate $G_n^2(\t)$ note
	first that for any $p>1,$ 
	\[
	P\left((K_{n}^{i})^{c}\cap D_{n}^{i}\right )\leq P\left (|\Delta_i^n X^c|>2v_n/3 \right )= O(\Delta_n^{\eps p}).
	\]
	Hence, by using H\"older's inequality, sub-polynomial growth of $f$, $(3)$ of Lemma \ref{lem:ergodic} and 
	$(3)$ of Lemma \ref{lem:jac-prot} we obtain for any $p>1,$
	\begin{align}\label{al:gn2}
	&\E\sup_{\t\in\Theta}|G_n^2(\t)|=\E\sup_{\t\in\Theta}|\sum_{i=1}^{n}f(\theta,X_{t_{i-1}})\left(\Delta_{i}^nX^{c}\right)1_{\left(K_{n}^{i}\right)^{c}\cap D_{n}^{i}}|=O(n\Delta_n^{\eps p}).
	\end{align}
	To estimate $G_n^3(\t)$
	note first that
	\begin{align*} 
	&P ((D_{n}^{i})^c)=P(|\Delta_i^nX^J| > v_n/3) \\
	&\leq P(|\int_{t_{i-1}}^{t_i}\int_{|z|\geq v_n}\gamma(X_{s-})z\mu(ds,dz)|
	 >  v_n/6)+P(|\int_{t_{i-1}}^{t_i}\int_{|z|<v_n}\gamma(X_{s-})z\mu(ds,dz)|
	 > v_n/6)
	\\
	&\leq P(|\int_{t_{i-1}}^{t_i}\int_{|z|\geq v_n}\gamma(X_{s-})z\mu(ds,dz)|>0)+\frac{ 6 }{v_n} E[|\int_{t_{i-1}}^{t_i}\int_{|z|<v_n}\gamma(X_{s})zds\nu(dz)|] 
	\\
	&
	 \leq P(\int_{t_{i-1}}^{t_i}\int_{|z|\geq v_n} \mu(ds,dz) \ge 1 )+\frac{ 6}{v_n} E[|\int_{t_{i-1}}^{t_i}\int_{|z|<v_n}\gamma(X_{s})zds\nu(dz)|]
	\\
	&
	 = 1-\exp\left( -\int_{t_{i-1}}^{t_i}\int_{|z|\geq v_n} ds \nu(dz) \right)+\frac{ 6 }{v_n} E[|\int_{t_{i-1}}^{t_i}\int_{|z|<v_n}\gamma(X_{s})zds\nu(dz)|]
	\\
	&
	=O\left( \Delta_n (\int_{|z|\geq v_n}\nu(dz)+\frac{1}{v_n}\int_{|z|<v_n}|z|\nu(dz))\right).
	\end{align*}
	Hence, using Hölder's inequality, the assumptions on $f$ and $(3)$ of Lemma \ref{lem:jac-prot} we obtain for any $q>1$ that
	\begin{align}\label{al:gn3}
	&\E\sup_{\t\in\Theta}|G_n^3(\t)|\leq \E \sum_{i=1}^n \sup_{\t\in\Theta}|f(\t,X_{t_{i-1}}  )  ||\Delta_i^nX^c|1_{\{|\Delta_i^n X|>v_n, (D_n^{i})^c\}}\leq O(n\Delta_n^{1/2})P((D_{n}^{i})^c)^{1/q}\nonumber \\
	&\leq O(n\Delta_n^{1/2})\Delta_n^{1/q} \left( \int_{| z  |\geq v_n}\nu(dz)+\frac 1{v_n}\int_{|z|<v_n}|z|\nu(dz) \right)^{1/q}.
\end{align}
	
	Finally, choosing $q^{-1}=1-\eps/2,$ we get from \eqref{eq:gn1}, \eqref{al:gn2} and \eqref{al:gn3} that
	\begin{align*}
	&\E \left[\sup_{\t\in\Theta} \left| \sum_{i=1}^{n}f(\theta,X_{t_{i-1}})\left(\Delta_{i}^nX^{c}-\Delta_{i}^nX\mathbf{1}_{|\Delta_{i}^nX|\leq v_n}\right)\right|\right] \leq O(n\Delta_{n}\int_{|z|\leq 2v_n/\gamma_{min}}|z|\nu(dz)) \\
	&+ O(n\Delta_{n}^{5/2-2\eps})+O(n\Delta_n^{3/2-\eps/2}) (\int_{|  z  |\geq v_n}\nu(dz)+\frac 1{v_n}\int_{|z|<v_n}|z|\nu(dz))^{1-\eps/2}
	\end{align*}
	In particular, using the definition of $v_n,$ finiteness of $\nu$ and of its first moment we immediately get
	\[
	\E\sup_{\t\in\Theta} \left| \sum_{i=1}^{n}f(\theta,X_{t_{i-1}})\left(\Delta_{i}^nX^{c}-\Delta_{i}^nX\mathbf{1}_{|\Delta_{i}^nX|\leq v_n}\right) \right|=
	O(n\Delta_n^{3/2-  \eps/2 }),
	\]
	hence  (i) is proved. \\
	To prove  (ii) 
	we decompose the approximation by the jump filter as follows:
	\begin{equation}\label{eq:decomp_Ant}
	\sum_{i=1}^{n}f(\theta,X_{t_{i-1}})\left(\Delta_{i}^nX^{c}-\Delta_{i}^nX\mathbf{1}_{|\Delta_{i}^nX|\leq v_n}\right)=
	G^{1}_n(\t)+A^{2}_n(\t)+A^{3}_n(\t),
	\end{equation}
	where $G^{1}_n(\t)$ is given by \eqref{al:G123} and
	\begin{align}\label{al:defAnt}
	& A^{2}_n(\t):=\sum_{i=1}^{n}f(\theta,X_{t_{i-1}})\left(\Delta_{i}^nX^{c}\right)1_{\left(K_{n}^{i}\right)^{c}\cap \left(M_{n}^{i}\right){}^{c}}\\
	\nonumber
	&A^{3}_n(\t):=\sum_{i=1}^{n}f(\theta,X_{t_{i-1}})\left(\Delta_{i}^nX^{c}\right)1_{\left(K_{n}^{i}\right)^{c}\cap M_{n}^{i}}.
	\end{align}
		Observe that
	\begin{align}
	A^{2}_n(\t)  =\sum_{i=1}^{n}f(\theta,X_{t_{i-1}})\Delta_{i}^nX^{c}\mathbf{1}_{(M_{n}^{i})^{c}}-
	\sum_{i=1}^{n}f(\theta,X_{t_{i-1}})\Delta_{i}^nX^{c}\mathbf{1}_{\{|\Delta_{i}^nX|
		 \leq  v_n\}\cap(M_{n}^{i})^{c}}.\label{eq:G2decomp}
	\end{align}
		We first show that after suitable renormalization the first term of this decomposition converges to zero in probability.
	Let $e_{i}:=f(\theta,X_{t_{i-1}})\Delta_{i}^nX^{c}\mathbf{1}_{(M_{n}^{i})^{c}}$.
	Denote ${\mathcal F}_{i}=\sigma\{ (W_s)_{0<s\leq t_i}, (L_s)_{0<s\leq t_i}, X_0\},$ then
	\begin{align*}
	E[e_{i}|\mathcal{F}_{i-1}] & =f(\theta,X_{t_{i-1}})E\left[\int_{t_{i-1}}^{t_{i}}\sigma(X_{s})dW_{s}1_{(M_{n}^{i})^{c}}|\mathcal{F}_{i-1}\right]\\
	& +f(\theta,X_{t_{i-1}})E\left[\int_{t_{i-1}}^{t_{i}}b(\ts  ,X_{s})ds1_{(M_{n}^{i})^{c}}|\mathcal{F}_{i-1}\right]
	\end{align*}
	Observe that $(W_{s})_{s\geq0}$ remains a Brownian motion with respect
	to the filtration that is enlarged by $\sigma(L)$, since $L$ and
	$W$ are independent. Therefore,
	\begin{equation*}
	E\left[\int_{t_{i-1}}^{t_{i}}\sigma(X_{s})dW_{s}1_{(M_{n}^{i})^{c}}|\mathcal{F}_{i-1}\right]=E\left[1_{(M_{n}^{i})^{c}}E\left[\int_{t_{i-1}}^{t_{i}}\sigma(X_{s})dW_{s}|\mathcal{F}_{i-1}\vee
	\sigma(L)\right]|\mathcal{F}_{i-1}\right]=0
	\end{equation*}
	and so
	\begin{align}\label{al:eibound}
	|E[e_{i}|\mathcal{F}_{i-1}]| & \leq|f(\theta,X_{t_{i-1}})|\int_{t_{i-1}}^{t_{i}}E\left[\left|b( \ts  ,X_{s})\right|1_{(M_{n}^{i})^{c}}|\mathcal{F}_{i-1}\right]ds.
	\end{align}
	Recall that
	\begin{equation*}
	P\left((M_n^i)^c\right)=1-P(\Delta_i^nN=0)=O(\Delta_n).
	\end{equation*}
	Using Hölder inequality, Lipshitz continuity of $b(\ts,. )$, the continuity of its Lipshitz constant given by the Assumption \ref{ass:existence} and
	Lemma \ref{lem:jac-prot} (2) we can write
	 for $p$, $q$ such that $p^{-1}+q^{-1}=1$, $p\ge 2$ and $C>0,$
	\begin{align}\label{al: fcm}
	&E\left[\left|b( \ts ,X_{s})\right|1_{(M_{n}^{i})^{c}}|\mathcal{F}_{i-1}\right]\leq \left(E\left[\left|b( \ts  ,X_{s})\right|^p|\mathcal{F}_{i-1}\right]\right)^{1/p}\Delta_n^{1/q}\nonumber\\
	&\leq C \left(  \E[|b( \ts,X_s)-b( \ts,X_{t_{i-1}})|^p|\F_{t_{i-1}}]+|b( \ts,X_{t_{i-1}})|^{ p }\right)^{1/p}    \Delta_n^{1/q}\nonumber\\
	&\leq  C\left(\left( \E\left[|X_s-X_{t_{i-1}}|^p|\F_{t_{i-1}}\right]\right)^{1/p}+
	|b( \ts ,X_{t_{i-1}})|
	\right) \Delta_n^{1/q}\nonumber\\
	&\leq  C 
	\Delta_n^{1/q}\left(\Delta_n^{1/p}(1+|X_{t_{i-1}}|^p)^{1/p}+|b( \ts ,X_{t_{i-1}})|\right).
	\end{align}
	
	Using the fact that $b(\ts ,.)$ and $\sup_{\t\in\Theta} |f(\t,.)|$  are sub-polynomial and
	choosing again $1/q=1-\eps/2,$ (which also guarantees $p>2,$)  we obtain
	\begin{equation}\label{eq:eifinbound}
	|E[e_{i}|\mathcal{F}_{i-1}]|\leq h(|X_{t_{i-1}}|)\Delta_{n}^{2-\eps/2},
	\end{equation}
	where $h$ is a polynomial function.
	Finally
	this implies that under the condition $n\Delta_n^{3-\eps}\to 0,$
	\begin{equation}
	E\left[\sum_{i=1}^{n}\left|E\left[\frac{e_{i}}{\sqrt{n\Delta_{n}}}|\mathcal{F}_{i-1}\right]\right|\right]=O(n^{1/2} \Delta_{n}^{3/2-\eps/2})\to0. \label{eq:cond_GCJ1}
	\end{equation}
	Next, we bound the moment of order two of $e_{i}$.
	
	By Hölder's inequality with $1/q=1-\eps/2$, $1/p=1-1/q$,   we have
	\begin{align}\label{al:ei2}
	E[e_{i}^{2}] & \le E\left[f(\theta,X_{t_{i-1}})^{2p}(\Delta_{i}^nX^{c})^{2p}\right]^{1/p}
	P\left[ (M_{n}^{i})^{c}\right]^{1/q}\\ \nonumber
	&  \le \Delta_n P\left[ (M_{n}^{i})^{c}\right]^{1-\varepsilon/2}=O(\Delta_n^{2-\varepsilon/2}),
	\end{align}
	where in the last line we used again Hölder's inequality, the sub-linear growth of $f$, together with Lemma \ref{lem:jac-prot} (3).
	Hence,
	\begin{equation}
	E\left[
	\left| \sum_{i=1}^{n}E\left[\left(\frac{e_{i}}{\sqrt{n\Delta_{n}}}\right)^{2}|\mathcal{F}_{i}\right]
	\right|
	\right]=
	\sum_{i=1}^{n}E\left[\left(\frac{e_{i}}{\sqrt{n\Delta_{n}}}\right)^{2}\right]= 
	O(\Delta_{n}^{1-\varepsilon/2})\to0.\label{eq:cond_GCJ2}
	\end{equation}
	Under (\ref{eq:cond_GCJ1}) and (\ref{eq:cond_GCJ2}) we obtain from Lemma 9
	in \cite{Genon1993} that
	\begin{equation}\label{eq:G2_decomp1}
	\frac{1}{\sqrt{n\Delta_{n}}}\sum_{i=1}^{n}f(\theta,X_{t_{i-1}})\Delta_{i}^nX^{c}\mathbf{1}_{(M_{n}^{i})^{c}}
	=\sum_{i=1}^{n} \frac{e_{i}}{\sqrt{n\Delta_{n}}}
	\stackrel{P}{\longrightarrow} 0
	\end{equation}
	if $n\Delta_{n}^{3- \varepsilon }\to0$.
	Recall that the second term in the decomposition   (\ref{eq:G2decomp}) of $A_n^{2}$ is given by
		\[
	\sum_{i=1}^{n}f(\theta,X_{t_{i-1}})\Delta_{i}^nX^{c}\mathbf{1}_{K_n^{i}\cap(M_{n}^{i})^{c}}.
	\]
	We will now bound this term in $\LL^1$. 
	%%%
	 We use again the decomposition \eqref{eq:indicator}
	of $K_n^{i}\cap(M_{n}^{i})^{c}$.
	We find that by computations similar to %estimate as for 
	\eqref{eq:indi3} and \eqref{eq:indi1}
	% (\ref{eq:indi2}) 
	respectively, we have
	\begin{align}
	\label{eq:lastdisp1}
	&
	E\left|\sum_{i=1}^{n}f(\theta,X_{t_{i-1}})\Delta_{i}^nX^{c}\mathbf{1}_{
		\{\Delta_{i}^nN\geq2\} \cap K_n^i}\right|=O(n\Delta_{n}^{5/2-2\eps}),\\
	\label{eq:lastdisp2}
	&
	E\left|\sum_{i=1}^{n}f(\theta,X_{t_{i-1}})\Delta_{i}^nX^{c}\mathbf{1}_{
		\{\Delta_{i}^nN=1,|\Delta_{i}^nL|\geq 2v_{n}/{\gamma_{min}}\}\cap K_n^i}\right|=O(n\Delta_{n}^{\eps p})
	\end{align}
	Moreover, we have that $P(\Delta^n_i N=1, |\Delta_{i}^nL|<2v_{n}/{\gamma_{min}})
	= P( \int_{t_{i-1}}^{t_i} \int_{|z| < 2v_{n}/{\gamma_{min}}} \mu(ds,dz) = 1)
	\le \Delta_n \int_{|z| < 2v_{n}}\nu(dz)$, where we used  $\gamma_{min}\geq 1$.
	From this, we can easily get
	\begin{equation} \label{eq:lastdisp3}
	E\left|\sum_{i=1}^{n}f(\theta,X_{t_{i-1}})\Delta_{i}^nX^{c}\mathbf{1}_{\{\Delta_{i}^nN=1,|\Delta_{i}^nL|<2v_{n}/{\gamma_{min}}\}\cap K_n^i}\right|=O\left(n\Delta_{n}^{3/2-\eps/2}\left(\int_{|z| < 2v_n}\nu(dz)\right)^{1-\eps/2}\right) .
	\end{equation}
	%%%%%%
	From \eqref{eq:G2decomp}, \eqref{eq:G2_decomp1}, \eqref{eq:lastdisp1}--\eqref{eq:lastdisp3}, we deduce that if $n\Delta_{n}^{3-\varepsilon}\to0$,

	\begin{equation}\label{eq:an2}
	A_n^{2}(\t)=o_{P}(\sqrt{n\Delta_n})+O_{
		 \LL^1 }\left(n\Delta_n^{5/2-2\eps}+\left( \int_{|z|<2v_n}\nu(dz)
	\right)^{1-\eps/2}
	n\Delta_n^{3/2-\eps/2}\right).
	\end{equation}
	It follows immediately
	from  Lemma \ref{lem:jac-prot} (3)  that for any $p>1,$
	\[  P ((K_n^i)^c\cap M_n^i)\leq  P (|\Delta_i^nX^c|>v_n)=O(\Delta_n^{\eps p}).
	\]
	Hence, using again Hölder's inequality and Lemma \ref{lem:jac-prot} (3) again, we see that for any $p>1,$
	\begin{equation}\label{eq:an3}
	E\left|A_n^{3}(\t)\right|=E\left|\sum_{i=1}^{n}f(\theta,X_{t_{i-1}})\left(\Delta_{i}^nX^{c}-\Delta_{i}^nX\mathbf{1}_{|\Delta_{i}^nX|\leq v_n}\right)1_{\left(K_{n}^{i}\right)^{c}\cap M_{n}^{i}}\right|=O(n\Delta_{n}^{ \varepsilon p}).
	\end{equation}
	%%%
	Finally, from \eqref{eq:gn1}, \eqref{eq:an2}, \eqref{eq:an3} we obtain that for any $\t\in\Theta,$  if $n\Delta_{n}^{3-\varepsilon}\to0$,
	\begin{align*}
	&\sum_{i=1}^{n}f(\theta,X_{t_{i-1}})\left(\Delta_{i}^nX^{c}-\Delta_{i}^nX\mathbf{1}_{|\Delta_{i}^nX|\leq v_n}\right)=\\
	&o_{P}(\sqrt{n\Delta_n})+O_{\LL^1}\left(n\Delta_n^{5/2-2\eps}+
	\left( \int_{|z|<2v_n}\nu(dz) \right)^{1-\varepsilon/2} 
	n\Delta_n^{3/2-\eps/2}+n\Delta_{n}\int_{|z|\leq 2v_n}|z|\nu(dz)\right).
	\end{align*}
	%%%
	This proves (ii). 
\end{proof}

\begin{proof}[Proof of Lemma \ref {lem:filter_infiniteactivity}]
	
	We start by proving $(i)$.
	In the infinite jump activity case, the L\'evy process has infinite number of jumps on all compact intervals. Hence, it is impossible to introduce the events that the process had no jump, one jump, or more than two jumps on $(t_{i-1},t_i]$ as it was done in the proof of Lemma \ref{lem:filter_error_finite}.

	Here,  we define the event on which  all the jumps of $L$ are small :
	\begin{equation}\label{eq:Nni}
	N_{n}^{i}=\left\{|\Delta  L_s| \leq 3v_n/\gamma_{min}; \   \forall s\in (t_{i-1},t_i] \right\},
	\end{equation}
	where $\Delta L_s:=L_s-L_{s-}.$
	Using the sets $K_n^{i}$ and $D_n^{i}$ from \eqref{eq:kni}, we define 
	\begin{equation}\label{eq:bn12}
	B^{1}_n(\t):=\sum_{i=1}^{n}f(\theta,X_{t_{i-1}})\left(\Delta_i^nX^J\right)1_{K_{n}^{i}\cap {(N_n^{i})}^c}; \quad B^{2}_n(\t):=
	\sum_{i=1}^{n}  
	f(\theta,X_{t_{i-1}})
	 \left(\Delta_i^nX^J\right)
	1_{K_{n}^{i}\cap N_n^{i}};\\
	\end{equation}
	and decompose the difference as follows
	\begin{equation}\label{eq:decb}
	\sum_{i=1}^{n}f(\theta,X_{t_{i-1}})\left(\Delta_{i}^nX^{c}-\Delta_{i}X\mathbf{1}_{|\Delta_{i}^nX|\leq v_n}\right)=B^{1}_n(\t)+B^{2}_n(\t)
	 +G_n^2(\theta)+ G_n^3(\theta),
		\end{equation}
		%%%
	 where $G_n^2(\theta)$ and $G_n^3(\theta)$ are defined in \eqref{E:Gn2}--\eqref{E:Gn3}. 
	%%%
		We start by  studying  the convergence of $B_n^1(\t)$.
	Let $T_i^*\in  (  t_{i-1};t_i]$ such that $|\Delta  L_{T_i^*}  |=\max \left \{|\Delta  L_{s}  |;\ s\in  ( t_{i-1};t_i]\right\}$.
	 Remark that $T_i^*$ is well defined, as from Assumption \ref{ass:jumps} (iii) there is, almost surely, a 
	unique time at which the L\'evy process admits a jump with maximal size. 
 We introduce  the event
	\begin{equation}\label{eq:Ani}
	A_n^{i}=\left\{ \sum_{t_{i-1}<s\leq t_i; s\neq T_i^*}  | \Delta L_s  | \leq \frac{v_n}{
		\gamma_{max} }\right\},
	\end{equation}
	 where $\gamma_{max}$ is defined in Assumption \ref{ass:jumps} (iv).
	
	To estimate  $B_n^1(\t)$ we make the decomposition
	\begin{equation*}
	K_{n}^{i}\cap {(N_n^{i})}^c=K_{n}^{i}\cap {(N_n^{i})}^c\cap A_n^{i} \quad \cup\quad K_{n}^{i}\cap {(N_n^{i})}^c\cap (A_n^{i})^c.
	\end{equation*}
	Note that 
	\begin{align*}
	& K_{n}^{i}\cap {(N_n^{i})}^c\cap A_n^{i}\\ 
	&\subset \{ |\Delta_i^nX^c+\gamma(X_{T_i^*-})\Delta L_{T_i^*}+\sum_{s\neq T^*_i}\Delta X_s |\leq v_n;\quad 
	|\gamma(X_{T_i^*-})\Delta L_{T_i^*}|>3v_n;\quad |\sum_{s\neq T^*_i}\Delta X_s|\leq v_n  \}\\
	&\subset \left\{ | \Delta_i^n X^c|\geq v_n\right\}.
	\end{align*}
	Hence, using $(3)$ from Lemma \ref{lem:jac-prot} we get for all $ p  >1:$
	\begin{equation}\label{eq:knabound}
	P(K_{n}^{i}\cap {(N_n^{i})}^c\cap A_n^{i})\leq P\left(|\Delta_{i}^nX^{c}|\geq v_n\right)=O(\Delta_{n}^{ p  /2}v_n^{- p })=O(\Delta_{n}^{  \varepsilon p }).
	\end{equation}
	Together with \eqref{al:jumpintest}, which is still true in the infinite activity case, Hölder's inequality, sub-polynomial growth  of $f$ and $(3)$ from Lemma \ref{lem:ergodic} this gives for any $p>1$ that
	\begin{equation}\label{eq:Bn11}
	E\sup_{\t\in\Theta}|\sum_{i=1}^{n}f(\theta,X_{t_{i-1}})\left(\Delta_i^nX^J\right)\left (1_{K_{n}^{i}\cap {(N_n^{i})}^c\cap A_n^{i}}\right )|=O( n\Delta_n^{\eps p}). 
	\end{equation}

	Using Hölder inequality, sub-polynomial growth of $f$,  Lemma \ref{lem:ergodic} (3),
	and Lemma \ref{L:law_sum_cond_big}, we get 
	for $1/p+1/q=1$ and some $C>0,$ 
	\begin{align}\label{al:Bn12}
	&\E\sup_{\t\in\Theta}|\sum_{i=1}^nf(\theta,X_{t_{i-1}})\left(\Delta_{i}^nX^{c}-\Delta_{i}^nX\mathbf{1}_{|\Delta_{i}^nX|\leq v_n}\right)1_{K_{n}^{i}\cap {(N_n^{i})}^c\cap (A_n^{i})^c}|\nonumber \\
	& \leq \sum_{i=1}^n\left (\E\sup_{\t\in\Theta}|f(\theta,X_{t_{i-1}})|^p(|\Delta_i^n X^c|+v_n)^p\right )^{1/p}\left (  P({(N_n^{i})}^c\cap (A_n^{i})^c)  \right )^{1/q} \nonumber 
		\\ \nonumber
	&  \leq C n  (\Delta_n^{1/2}+ v_n) \left( \frac{\Delta_n^2}{v_n}\int_{|z|\geq 3 v_n/\gamma_{min}}\nu(dz) \right)^{1/q} 
	\\
	&
	\leq C n v_n^{\varepsilon/2} \Delta_n^{2-\varepsilon}\left (\int_{|z|  \geq 3 v_n /\gamma_{min}
		 }\nu(dz) \right )^{1-\varepsilon/2}, \quad \text {choosing $1/q=1-\varepsilon/2$.}
	\end{align}
		 From \eqref{eq:Bn11} and \eqref{al:Bn12}, we get
	\begin{align}
	E\sup_{\t\in\Theta}|B_n^1(\t)|=o\left(n\Delta_n^{3/2-\eps/2} \left(\int_{|z|\geq v_n / \gamma_{min} }\nu(dz)\right)^{1-\eps/2}\right). \label{eq:Bn1}
	\end{align}
	%\begin{align}
	%&E\sup_{\t\in\Theta}|B_n^1(\t)|\leq n\Delta_n^{5/2-\eps}\left (\frac{\int_{|z|>v_n}\nu(dz)}{v_n} \right )^{1-\eps/2}+C n\Delta_n^{\eps p} \nonumber \\
	%&=o\left(n\Delta_n^{3/2-\eps/2} \left(\int_{|z_n|\geq v_n}\nu(dz)\right)^{1-\eps/2}\right) \label{eq:Bn1} 
	%\end{align}
	To estimate $B_n^2(\t)$ we use the bound
	\begin{align*}
	&\sum_{i=1}^nE|f(\theta,X_{t_{i-1}})\int_{t_{i-1}}^{t_i}\int_{\R\setminus \{ 0\}}\gamma(X_{s-})z\mu(ds,dz)|1_{K_{n}^{i}\cap {N_n^{i}}} \\
	& \leq \sum_{i=1}^nE 
	\int_{t_{i-1}}^{t_i}\int_{|z|\leq 3v_n/\gamma_{min} }
	|f(\theta,X_{t_{i-1}})\gamma(X_{s-})z|\mu(ds,dz)
	\\
	& \leq \sum_{i=1}^n
	\int_{t_{i-1}}^{t_i} \int_{|z|\leq 3v_n/\gamma_{min} }
	E [|f(\theta,X_{t_{i-1}})\gamma(X_{s})|]|z|\nu(dz)ds =  O( n \Delta_n \int_{|z|\leq 3v_n/\gamma_{min}}|z|\nu(dz)).
	\end{align*}
	Since $\gamma_{min}\geq 1$, we obtain,
	% we obtain
	\begin{equation}\label{eq:Bn2}
	 E\sup_{\t\in\Theta}|B_n^2(\t)|
	=O( n \Delta_n \int_{|z|\leq 3v_n/\gamma_{min} }|z|\nu(dz))\leq O( n \Delta_n \int_{|z|\leq 3v_n }|z|\nu(dz)).
	\end{equation}

	The $\LL^1$ norms of $\sup_{\theta \in \Theta} |G_n^2(\theta)|$ and $\sup_{\theta \in \Theta} |G_n^3(\theta)|$ have been studied in the Lemma \ref{lem:filter_error_finite}, when the L\'evy process has finite activity. 
	However, the proofs of the upper bounds \eqref{al:gn2} and \eqref{al:gn3}, obtained in Lemma \ref{lem:filter_error_finite}, do not use the fact that 
	$\nu(\R)<\infty.$
	
	Finally, collecting  \eqref{al:gn2}, \eqref{al:gn3} with $1/q=1-\varepsilon/2$, \eqref{eq:Bn1}, and \eqref{eq:Bn2} we obtain $(i)$.
	 We continue with the proof of $(ii)$.
	%%%%%%%%%%%%%%%%%%%%%%%%%%%%%%%%
	Using the events $K_n^i$ and $ N_n^i$ given by \eqref{eq:kni} and \eqref{eq:Nni} we
	define 
	
	\begin{equation*}
	B^{3}_n(\t):=\sum_{i=1}^{n}f(\theta,X_{t_{i-1}})\left(\Delta_i^nX^c\right)1_{(K_{n}^{i})^c\cap (N_n^{i})^c}; \quad B^{4}_n(\t):=
	\sum_{i=1}^{n}  
	f(\theta,X_{t_{i-1}})
	 \left(\Delta_i^nX^c\right)
	1_{(K_{n}^{i})^c\cap (N_n^{i})};\\
	\end{equation*}
	and decompose the difference as follows
	\begin{equation}\label{eq:decb}	\sum_{i=1}^{n}f(\theta,X_{t_{i-1}})\left(\Delta_{i}^nX^{c}-\Delta_{i}X\mathbf{1}_{|\Delta_{i}^nX|\leq v_n}\right)=B^{1}_n(\t)+B^{2}_n(\t)
	 +B_n^3(\theta)+ B_n^4(\theta),
		\end{equation}
	where $B^{1}_n(\t)$ and $B^{2}_n(\t)$ are given by \eqref{eq:bn12}.
	Using \eqref{eq:Bn11} and \eqref{al:Bn12} we can see that
	\begin{equation}\label{eq:boundbn1}
	E|B_n^1(\t)|=o\left(n\Delta_n^{2-\eps} \left(\int_{|z|\geq 3v_n / \gamma_{min} }\nu(dz)\right)^{1-\eps/2}\right),
	\end{equation}
	while \eqref{eq:Bn2} gives the bound for $E|B_n^2(\theta)|.$
	The role of the event $N_n^{i}$ (all the jumps of $L$ are small)  in the case of the infinite activity is similar to the role of $M_n^{i}$ ($L$ does not jump)  in the finite activity case. Therefore,
	to estimate $B_n^3(\t)$ we use a decomposition similar to \eqref{eq:G2decomp}, where we replace $M_n^{i}$ by $N_n^{i}$ which leads to
		\begin{equation}\label{eq:Bn3dec}
	B_n^3(\t)=\sum_{i=1}^{n}f(\theta,X_{t_{i-1}})\left(\Delta_i^nX^c\right)1_{(N_n^{i})^c}-\sum_{i=1}^{n}f(\theta,X_{t_{i-1}})\left(\Delta_i^nX^c\right)1_{K_{n}^{i}\cap (N_n^{i})^c}.
	\end{equation}
	We will show that the first term of this decomposition goes to zero after suitable normalization.
%%%%%%%%%%%%%%%
Let $\tilde e_{i}:=f(\theta,X_{t_{i-1}})\Delta_{i}^nX^{c}\mathbf{1}_{(N_{n}^{i})^{c}}$.
	Recall that
	\begin{align*}
	P\left((N_n^i)^c\right)=1-P(\int_{t_{i-1}}^{t_i}\int_{|z|>3v_n/\gamma_{min}}\mu(ds,dz)=0)=1-e^{-\Delta_n\int_{|z|>3v_n/\gamma_{min}}\nu(dz)}\\ =O\left (\Delta_n\int_{|z|>3v_n/\gamma_{min}}\nu(dz)\right).
	\end{align*}
	Therefore, the same arguments that were used to obtain \eqref{eq:eifinbound} give here
	\begin{equation}\label{eq:tildeeibound}
	|E[\tilde e_{i}|\mathcal{F}_{i-1}]|\leq h(|X_{t_{i-1}}|)\Delta_{n}^{2-\eps/2}\left (\int_{|z|>3v_n/\gamma_{min}}\nu(dz)\right )^{1-\eps/2},
	\end{equation}
	where $h$ is a polynomial function.
	Hence, under the condition $n\Delta_n^{3-\eps}\left (\int_{|z|>3v_n/\gamma_{min}}\nu(dz)\right )^{2-\eps}\to 0,$
	\begin{equation}
	E\left[\sum_{i=1}^{n}\left|E\left[\frac{\tilde e_{i}}{\sqrt{n\Delta_{n}}}|\mathcal{F}_{i-1}\right]\right|\right]=O\left (n^{1/2} \Delta_{n}^{3/2-\eps/2}\left (\int_{|z|>3v_n/\gamma_{min}}\nu(dz)\right )^{1-\eps/2}\right)\to0. \label{eq:like63}
	\end{equation}

	Next, we bound the second moment of  $\tilde e_{i}$. Similarly to \eqref{al:ei2} we obtain

	\begin{equation}
	E[\tilde e_{i}^{2}] \leq \Delta_nP[(N_n^{i})^c]^{1-\eps/2}=O\left(\Delta_n^{2-\varepsilon/2}\left (\int_{|z|>3v_n/\gamma_{min}}\nu(dz)\right )^{1-\eps/2}\right),
	\end{equation}
	Hence, using  $\Delta_n\int_{|z|>3v_n/\gamma_{min}}\nu(dz)\to 0,$ which is implied by  $n\Delta_n^{3-\eps}\left (\int_{|z|>3v_n/\gamma_{min}}\nu(dz)\right )^{2-\eps}\to 0,$ we have
	\begin{align}
	E\left[
	\left| \sum_{i=1}^{n}E\left[\left(\frac{\tilde e_{i}}{\sqrt{n\Delta_{n}}}\right)^{2}|\mathcal{F}_{i}\right]
	\right|
	\right]=
	\sum_{i=1}^{n}E\left[\left(\frac{\tilde e_{i}}{\sqrt{n\Delta_{n}}}\right)^{2}\right] \nonumber\\
	= O\left(\Delta_{n}^{1-\varepsilon/2}\left (\int_{|z|>3v_n/\gamma_{min}}\nu(dz)\right )^{1-\eps/2}\right)\to0.\label{eq:like64}
	\end{align}
	Under (\ref{eq:like63}) and (\ref{eq:like64}) we obtain from Lemma 9
	in \cite{Genon1993} that
	\begin{equation}\label{eq:B3_decomp1}
	\frac{1}{\sqrt{n\Delta_{n}}}\sum_{i=1}^{n}f(\theta,X_{t_{i-1}})\Delta_{i}^nX^{c}\mathbf{1}_{(N_{n}^{i})^{c}}
	=\sum_{i=1}^{n} \frac{\tilde e_{i}}{\sqrt{n\Delta_{n}}}
	\stackrel{P}{\longrightarrow} 0
	\end{equation}
	if $n\Delta_n^{3-\eps}\left (\int_{|z|>3v_n/\gamma_{min}}\nu(dz)\right )^{2-\eps}\to 0.$

%%%%%%%%%%%%%%%%%%%%%%%%%%%
Recall that the second term in the decomposition   (\ref{eq:Bn3dec}) of $B_n^3$ is given by
		\[
	\sum_{i=1}^{n}f(\theta,X_{t_{i-1}})\Delta_{i}^nX^{c}\mathbf{1}_{K_n^{i}\cap(N_{n}^{i})^{c}}.
	\]
	We will now bound this term in $\LL^1$. Using the set $A_n^{i}$ defined by \eqref{eq:Ani} we decompose 
	\[\mathbf{1}_{K_n^{i}\cap(N_{n}^{i})^{c}}=\mathbf{1}_{K_n^{i}\cap(N_{n}^{i})^{c}\cap A_n^{i}}+\mathbf{1}_{K_n^{i}\cap(N_{n}^{i})^{c}\cap (A_n^{i})^c}.
	\]
	The first term of this decomposition is bounded in $\LL^1$ using \eqref{eq:knabound}. As a result, for all $p>1,$
		\begin{equation}\label{eq:Bn321}
	E\sup_{\t\in\Theta}|\sum_{i=1}^{n}f(\theta,X_{t_{i-1}})\left(\Delta_i^nX^c\right)\left (1_{K_{n}^{i}\cap {(N_n^{i})}^c\cap A_n^{i}}\right )|=O( n\Delta_n^{\eps p}). 
	\end{equation}
	Then, exactly as in \eqref{al:Bn12}, we get
	\begin{equation}\label{eq:Bn322}
	E|\sum_{i=1}^{n}f(\theta,X_{t_{i-1}})\left(\Delta_i^nX^c\right)\left (1_{K_{n}^{i}\cap {(N_n^{i})}^c\cap (A_n^{i})^c}\right )|=o\left( n\Delta_n^{2-\eps}(\int_{|z| \geq v_n/\gamma_{min}}\nu(dz))^{1-\eps/2}\right).
	\end{equation}
	As a result,
	\begin{equation}\label{eq:Bn3res}
	B_n^3(\t)=o_P(\sqrt{n\Delta_n})+o_{\LL^1}\left(n\Delta_n^{2-\eps}(\int_{|z| \geq v_n/\gamma_{min}}\nu(dz))^{1-\eps/2}\right ).
	\end{equation}
	It remains to estimate the term $B_n^4$ in the decomposition \eqref{eq:decb}.
	Observe that for all $p>1,$
	\begin{align*}	&
	P((K_n^{i})^c\cap N_n^{i})=\nonumber\\
	&P(|\Delta_i^nX^c+\sum_{t_{i-1}< s\leq t_i}\Delta X_s|>v_n; N_n^{i})\leq 
	P(|\Delta_i^nX^c|>\frac{v_n}{2})+P(|\sum_{t_{i-1} < s\leq t_i}\Delta X_s|>\frac{v_n}2;N_n^{i})\leq \nonumber\\
	&
	C\Delta_n^{\eps p}+P(|\int_{t_{i-1}}^{t_i}\gamma(X_{s-})\int_{|z|\leq 3v_n/\gamma_{min}}z\mu(ds,dz)|>\frac{v_n}2)\leq 
	C\Delta_n^{\eps p}+\frac{\Delta_n}{v_n}\int_{|z|\leq 3v_n/\gamma_{min}}|z|\nu(dz),\nonumber
	\end{align*}
	where $C>0.$
	Using Hölder's inequality twice, this last bound, sub-polynomial growth of $f$ and Lemma \ref{lem:jac-prot} (iii) we can easily see that with $1/q=1-\eps/2$ we get
	\begin{align}\label{al:Bn4bound}
	&E|B_n^4(\t)|=\\
	&E|\sum_{i=1}^{n}f(\theta,X_{t_{i-1}})\left(\Delta_i^nX^c\right)1_{(K_{n}^{i})^c\cap N_n^{i}}|\leq 
		\sum_{i=1}^{n}\left (E|f(\theta,X_{t_{i-1}})|^p|\Delta_i^nX^c|^p\right)^{1/p}P^{1/q}(K_{n}^{i})^c\cap N_n^{i})\leq\nonumber\\
		& n\Delta_n^{1/2}\left(\frac{\Delta_n}{v_n}\int_{|z|\leq 3v_n/\gamma_{min}}|z|\nu(dz)\right)^{1-\eps/2}\leq n\Delta_n^{1+\eps/2}\left (\int_{|z|\leq 3v_n/\gamma_{min}}|z|\nu(dz)\right)^{1-\eps/2}.\nonumber
\end{align}
Finally, collecting \eqref{eq:Bn2}, \eqref{eq:boundbn1}, \eqref {eq:Bn3res} and \eqref{al:Bn4bound} we obtain assertion $(ii)$ of the lemma.
	%%%
%	 We use again the decomposition \eqref{eq:indicator}
%	of $K_n^{i}\cap(M_{n}^{i})^{c}$.
%	We find that by computations similar to %estimate as for 
%	\eqref{eq:indi3} and \eqref{eq:indi1}
%	% (\ref{eq:indi2}) 
%	respectively, we have
%	\begin{align}
%	\label{eq:lastdisp1}
%	&
%	E\left|\sum_{i=1}^{n}f(\theta,X_{t_{i-1}})\Delta_{i}^nX^{c}\mathbf{1}_{
%		\{\Delta_{i}^nN\geq2\} \cap K_n^i}\right|=O(n\Delta_{n}^{5/2-2\eps}),\\
%	\label{eq:lastdisp2}
%	&
%	E\left|\sum_{i=1}^{n}f(\theta,X_{t_{i-1}})\Delta_{i}^nX^{c}\mathbf{1}_{
%		\{\Delta_{i}^nN=1,|\Delta_{i}^nL|\geq 2v_{n}/{\gamma_{min}}\}\cap K_n^i}\right|=O(n\Delta_{n}^{\eps p})
%	\end{align}
%	%%
%	Moreover, we have that $P(\Delta^n_i N=1, |\Delta_{i}^nL|<2v_{n}/{\gamma_{min}})
%	= P( \int_{t_{i-1}}^{t_i} \int_{|z| < 2v_{n}/{\gamma_{min}}} \mu(ds,dz) = 1)
%	\le \Delta_n \int_{|z| < 2v_{n}}\nu(dz)$, where we used  $\gamma_{min}\geq 1$.
%	From this, we can easily get
%	\begin{equation} \label{eq:lastdisp3}
%	E\left|\sum_{i=1}^{n}f(\theta,X_{t_{i-1}})\Delta_{i}^nX^{c}\mathbf{1}_{\{\Delta_{i}^nN=1,|\Delta_{i}^nL|<2v_{n}/{\gamma_{min}}\}\cap K_n^i}\right|=O\left(n\Delta_{n}^{3/2-\eps/2}\left(\int_{|z| < 2v_n}\nu(dz)\right)^{1-\eps/2}\right) .
%	\end{equation}
%	%%%%%%
%	From \eqref{eq:G2decomp}, \eqref{eq:G2_decomp1}, \eqref{eq:lastdisp1}--\eqref{eq:lastdisp3}, we deduce that if $n\Delta_{n}^{3-\varepsilon}\to0$,
%
%	\begin{equation}\label{eq:an2}
%	A_n^{2}(\t)=o_{P}(\sqrt{n\Delta_n})+O_{
%		 \LL^1 }\left(n\Delta_n^{5/2-2\eps}+\left( \int_{|z|<2v_n}\nu(dz)
%	\right)^{1-\eps/2}
%	n\Delta_n^{3/2-\eps/2}\right).
%	\end{equation}
\end{proof}

\begin{proof}[Proof of Lemma \ref{lem:euler}]
Using $dX_{s}^{c}=b(\ts,X_{s})ds+\sigma(X_{s})dW_{s}$ we decompose the difference as
\begin{equation}\label{eq:decomp}
\int_{0}^{t_{n}}f(\theta,X_{s})\; dX_{s}^{c}-\sum_{i=1}^{n}f(\theta,X_{t_{i-1}})\Delta_{i}^nX^{c}= A_{n,1}(\t)+A_{n,2}(\t)+A_{n,3}(\t),
\end{equation}
where
\begin{align}\label{al:decomp}
   A_{n,1}(\t):=&\sum_{i=1}^{n}\int_{t_{i-1}}^{t_{i}}(f(\theta,X_{s})-f(\theta,X_{t_{i-1}}))\sigma(X_{s})dW_{s},\\
   \nonumber
  A_{n,2}(\t):=&\sum_{i=1}^{n}\int_{t_{i-1}}^{t_{i}}(f(\theta,X_{s})-f(\theta,X_{t_{i-1}}))(b(\ts,X_{s})-b(\ts,X_{t_{i-1}}))ds,\\
  \label{al:decomp_terme3}
 A_{n,3}(\t):=&\sum_{i=1}^{n}\int_{t_{i-1}}^{t_{i}}(f(\theta,X_{s})-f(\theta,X_{t_{i-1}}))b(\ts,X_{t_{i-1}})ds.
  \end{align}
Let us start by proving (ii).
 Let as previously   $\F_t=\sigma\{X_0, W_u,L_u;\ u\leq t\}, t\geq 0.$ 
Using  martingale property  and Itô's isometry of the stochastic integral together with the finite increments formula applied to $f$, we obtain 
\begin{align*}
E[A^2_{n,1}(\t)] & =E\left[\sum_{i=1}^{n}\left(\int_{t_{i-1}}^{t_{i}}(f(\theta,X_{s})-f(\theta,X_{t_{i-1}}))\sigma(X_{s})dW_{s}\right)^{2}\right]\\
 & =E\sum_{i=1}^{n}\int_{t_{i-1}}^{t_{i}}(f(\theta,X_{s})-f(\theta,X_{t_{i-1}}))^{2}\sigma^{2}(X_{s})ds\\
 & \leq\sum_{i=1}^{n}\int_{t_{i-1}}^{t_{i}}E\left[(X_{s}-X_{t_{i-1}})^{2}f'^2(\theta,\tilde x)\sigma^{2}(X_{s})\right]ds,
\end{align*}
  where  $\tilde x$ is a point between $X_s$ and $X_{t_{i-1}}$. Note that $|\tilde x|\leq |X_s|+ |X_{t_{i-1}}|$. Using sub-polynomial growth of $\sigma$ and $\sup_{\t}|f'(\t, .)|,$  Hölder's inequality, (3) of the Lemma \ref{lem:ergodic} and  $(1)$ of the Lemma \ref{lem:jac-prot} yields
  %Using Lipshitz continuity of  $\sigma$ we can write:
\begin{equation}\label{eq:new}
E\left[(X_{s}-X_{t_{i-1}})^{2}f'^2(\theta,\tilde x)\sigma^{2}(X_{s})\right]  \leq C\E[|X_s-X_{t_{i-1}}|^{2q}]^{1/q}\leq C\Delta_n^{1/q},
\end{equation}
where $q>1$ and $C$ is a positive constant.
%&E\left[(X_{s}-X_{t_{i-1}})^{2}(\sigma(X_{s})-\sigma(X_{t_{i-1}}))^{2}\right] +
% 2E\left[(X_{s}-X_{t_{i-1}})^{2}\left(\sigma(X_{s})-\sigma(X_{t_{i-1}})\right)\sigma(X_{t_{i-1}})\right]
% +E\left[(X_{s}-X_{t_{i-1}})^{2}\sigma(X_{t_{i-1}})^2\right] \\
% &\leq C\left( \E \left[(X_{s}-X_{t_{i-1}})^{4}\right]+\E\left[ (X_{s}-X_{t_{i-1}})^{3}\sigma(X_{t_{i-1}})\right]+\E\left[ (X_{s}-X_{t_{i-1}})^{2}\sigma^2(X_{t_{i-1}})\right]\right)
%Conditioning on $\mathcal{F}_{t_{i-1}}$ and using $1)$ of the lemma \eqref{lem:jac-prot}, we finally get with some $p>0$:
%\begin{equation*}
%E\left[(X_{s}-X_{t_{i-1}})^{2}\sigma(X_{s})^{2}|\mathcal{F}_{t_{i-1}}\right]\leq\Delta_n\left(1+\sigma(X_{t_{i-1}})+\sigma^2(X_{t_{i-1}})\right)(2+|X_{t_{i-1}}|^{p})
%\end{equation*}
%such that using sub-linear growth of $\sigma$ we obtain
%\begin{equation*}
%E[(A_{n,1}(\t))^2]\leq C\sum_{i=1}^{n}\left[ \Delta_n^2E\left(1+\sigma(X_{t_{i-1}})+\sigma^2(X_{t_{i-1}})\right)(2+|X_{t_{i-1}}|^{p})
% \right]\leq Cn\Delta_{n}^{2}.
%\end{equation*}
Hence, for all $\t\in\Theta,$ 
\begin{equation*} %\label{eq:An1}
\E [A^2_{n,1}(\t)]\leq Cn\Delta_{n}^{1+1/q}
 \end{equation*}
 %and  by choosing $q$ close enough to $1$ we have $n \Delta_n^{1 + 2/q} \to 0$  
  and consequently 
\begin{equation}\label{eq:An1}
\frac1{\sqrt{n\Delta_n}}A_{n,1}(\t)\stackrel{L^2}{\longrightarrow} 0.
\end{equation}

Using Lipshitz continuity of $b$, and the same arguments than for obtaining \eqref{eq:new},  it follows immediately that 
\begin{equation}\label{eq:An22}
\E[\sup_{\t\in\Theta}|A_{n,2}(\t)|]\leq Cn\Delta_n^{1+1/q}
\end{equation}
  Hence,  by choosing $q=1-\eps/2$ such that  $n \Delta_n^{1 + 2/q}= n\Delta_n^{3-\eps} \to 0$ it follows that
 \begin{equation}\label{eq:An2}
\frac1{\sqrt{n\Delta_n}}\sup_{\t\in\Theta}|A_{n,2}(\t)|\stackrel{L^1}{\longrightarrow} 0.
\end{equation}
Observe that by Itô's formula $A_{n,3}(\t)$ can be written as 
\[
A_{n,3}(\t)=a_{n}(\t)+b_{n}(\t)+c_{n}(\t),
\]
where
\begin{align*}
a_n(\t) & =\sum_{i=1}^nb(\ts,X_{t_{i-1}})\int_{t_{i-1}}^{t_i}ds\int_{t_{i-1}}^{s}f'(\theta,X_{u})\sigma(X_u)dW_u,\\
b_n(\t) & =\sum_{i=1}^nb(\ts,X_{t_{i-1}})\int_{t_{i-1}}^{t_i}ds\int_{t_{i-1}}^{s}\left[f'(\t,X_u)b(\ts,X_u)+f''(\theta,X_{u})\frac{1}{2}\sigma^2(X_{u})\right]du,\\
c_n(\t) & =\sum_{i=1}^nb(\ts,X_{t_i-1})\int_{t_{i-1}}^{t_i}ds\sum_{\tau\in[t_{i-1},s]}(f(\theta,X_{\tau})-f(\theta,X_{\tau-})).
\end{align*}
%Since $dX_{s}^{c}=b(\ts,X_{s})ds+\sigma(X_{s})dW_{s}$ and the
%expectation of the martingale part vanishes, we obtain for the conditional
%expectation of $B_{1}$ with respect to $\mathcal{F}_{t_{i-1}}$
Denote
\[e_i^n:=
\frac{1}{\sqrt{n\Delta_n}}\int_{t_{i-1}}^{t_i} ds\int_{t_{i-1}}^{s}b(\ts,X_{t_{i-1}})f'(\theta,X_{u})\sigma(X_u)dW_u
\]
%\brc C'est la tribu de $X$ ou de $W$ ? Avant on l'a noté comme celle de $W$, je corrige  Let as previously $\F_s=\sigma\{W_u, u\leq s\}.$
Using martingale property of the stochastic integral with respect to $W$ we obtain
\[\E\left[{e_i^n}|\mathcal{F}_{t_{i-1}}\right]=0.\]
Using H\"older's inequality and  isometry property of the stochastic integral we get 
 \begin{align*}
\E\left (\E\left[(e_i^n)^2|\mathcal{F}_{t_{i-1}}\right]\right )=\E\left[(e_i^n)^2\right]&\leq \frac{1}n \int_{t_{i-1}}^{t_i}ds\E\left(\int_{t_{i-1}}^{s}b(\ts,X_{t_{i-1}})f'(\theta,X_{u})\sigma(X_u)dW_u\right)^2  \\
&=\frac{1}n\int_{t_{i-1}}^{t_i} ds\int_{t_{i-1}}^{s}\E\left[ b^2(\ts,X_{t_{i-1}})f'^2(\theta,X_{u})\sigma^2(X_u)\right]du 
\leq C\frac {\Delta_n^2}n,
%\frac{b^2(\ts,X_{t_{i-1}})}n\Delta_n^2
\end{align*} 
where in the last inequality we have used the uniform in $\t$ sub-polynomial growth of $f'$ and $b$, sub-linear growth of $\sigma$  and Lemma \ref{lem:ergodic}(3).
Therefore
% Hence, using sub-polynomial growth of $b$ given by the Assumpton \ref {ass:subpolynom}  and again $(3)$ from Lemma \ref{lem:ergodic}
\[\E\sum_{i=1}^n\E\left[\left(e_i^n\right)^2|\mathcal{F}_{t_{i-1}}\right]\leq C\Delta_n^2\to 0\quad\text{when}\quad n\to\infty.\]
We conclude, using Lemma 9 in \cite{Genon1993}, that $\forall \t\in\Theta$,
\begin{equation}\label{eq:a}
\frac1{\sqrt{n\Delta_n}}a_n(\t)=\sum_{i=1}^ne_i^n\stackrel{P}{\longrightarrow} 0.
\end{equation}
Using again uniform in $\t$ sub-polynomial growth  of $b, f', f''$, sub-linearity of $\sigma$ and $(3)$ of the Lemma \ref{lem:ergodic}  we easily see that
\begin{equation}\label{eq:b}
\E\sup_{\t\in\Theta}|b_n(\t)|\leq Cn\Delta_n^2.
\end{equation}  
Let us now derive a bound for the jump term $c_n$. 
\begin{align}\label{al:c}
&\E\sup_{\t\in\Theta}|c_n(\t)|\\ & \leq\sum_{i=1}^n\int_{t_{i-1}}^{t_i}ds\int_{t_{i-1}}^{s}du\int_{\mathbb{R}\setminus\{0\}}E|b(\ts,X_{t_{i-1}})||f(\theta,X_{u-}+\gamma(X_{u-})z)-f(\theta,X_{u-})|\mu(du,dz)\nonumber\\
 & \leq\sum_{i=1}^n\int_{t_{i-1}}^{t_i}ds\int_{t_{i-1}}^{s}du\int_{\mathbb{R}\setminus\{0\}}\E|b(\ts,X_{t_{i-1}})f'(\theta,\tilde x)\gamma(X_{u-})||z|
 \nu(dz).\nonumber\\
\end{align}
 where in the second inequality we used again the finite increments formula and  denoted $\tilde x  $ the corresponding  point between $X_{u-}$ and $X_u=X_{u-}+ \gamma(X_{u-})z.$ Note that  again $|\tilde x|\leq |X_{u-}|+|X_u|.$
 According to the Assumptions \ref{ass:ergodic} (i), (iii) and the assumption b) of the Lemma, the functions  $\gamma$,  $b(\ts,.)$ and $ \sup_{\t}|f'(\t, .)|$ are sub-polynomial, 
 and  $\nu(|z|)<\infty. $ Therefore, using  $(3)$ from Lemma \ref{lem:ergodic}
we have
\[\sup_{\t\in\Theta}\int_{\mathbb{R}\setminus\{0\}}\E|b(\ts,X_{t_{i-1}}) f'(\theta,\tilde x)\gamma(X_{u-})||z|
 \nu(dz)<\infty.\]  
This last inequality together with \eqref{al:c} gives
\begin{equation}\label{eq:cc}
 \E\sup_{\t\in\Theta}|c_n(\t)|=O( n\Delta_n^2).
\end{equation}
From \eqref{eq:a}, \eqref{eq:b} and \eqref{eq:cc} we conclude that under condition $n\Delta_n^ {3-\eps}\to 0,$
\begin{equation} \label{E:cv_term_A3}
 \frac1{\sqrt{n\Delta_n}}A_{n,3}(\t)\stackrel{P}{\longrightarrow} 0.
\end{equation} 
Finally, the previous display  together with \eqref{eq:An1} and \eqref{eq:An2} proves (ii) of the lemma. 
To prove the claim (i) we will again use the decomposition of the difference given by
\eqref{eq:decomp}.

Using the same arguments as in \eqref{eq:new} and Lemma \ref{lem:jac-prot} (1), we get for some $p>1$, $C>0$ and $\tilde x$ between $X_s$ and $X_{t_{i-1}}$:
\begin{align*}
&\E\sup_{\t\in\Theta} |A_{n,3}(\t)|\leq
C\sum_{i=1}^n\int_{t_{i-1}}^{t_i}\E\left [|f'(\t,\tilde x)(1+|X_{t_{i-1}}|^p)|\left|X_s-X_{t_{i-1}}\right|\right]ds\leq\nonumber\\
&C\sum_{i=1}^n\int_{t_{i-1}}^{t_i}\E\left(\left|X_s-X_{t_{i-1}}\right|^2\right)^{1/2}\left (\E\left [|f'(\t,\tilde x)|^2(1+|X_{t_{i-1}}|^{2p})\right]\right)^{1/2}ds\leq\nonumber\\
& \sum_{i=1}^n\int_{t_{i-1}}^{t_i}C\Delta_n^{1/2}ds
\leq Cn\Delta_n^{3/2}.
\end{align*}

Hence 
\begin{equation}\label{eq:AAn3}
\frac{1}{n\Delta_n}\sup_{\t\in\Theta}|A_{n,3}(\t)|\stackrel{L^1}{\longrightarrow} 0.
\end{equation}
The bound \eqref{eq:An22} gives
\begin{equation}\label{eq:AAn2}
\frac{1}{n\Delta_n}\sup_{\t\in\Theta}|A_{n,2}(\t)|\stackrel{L^1}{\longrightarrow} 0.
\end{equation}
From \eqref{eq:An1} we know that 
\[\forall \t\in\Theta,\quad \frac 1{n\Delta_n}A_{n,1}(\t)\stackrel{P}{\longrightarrow} 0.\] 
Let us prove that this convergence holds uniformly with respect to $\theta$.
 Denote $\phi:[0,t_n]\to[0,t_n],$ $\phi(s)=t_{i-1}$ if $t_{i-1}\leq s<t_i, i=0,\ldots, n-1,$
 and define
\[M_n(\t):=\frac{1}{t_n}A_{n,1}(\t)=\frac{1}{t_n}\int_{0}^{t_n}(f(\theta,X_{s})-f(\theta,X_{{\phi(s)}}))\sigma(X_{s})dW_{s}.\]
Using Burkholder-Davis-Gundy inequality, Hölder continuity of $f$, sub-polynomial growth of its Hölder constant $K$, sub-linear growth of $\sigma$ and the boundedness  of moments of $X$ given by $(3)$ of Lemma \ref{lem:ergodic} we find that for any $p\geq 2$ and some $C>0,$

\begin{align*}
\E|M_n(\t)-M_n(\t')|^p &\leq
|\t-\t'|^{\kappa p}\frac C{t_n^{p/2}}\E\left(\frac 1{t_n}\int_0^{t_n}\left(K^2(X_s)+K^2(X_{{\phi(s)}})\right)\sigma(X_{s})^2ds\right )^{p/2}\\ 
&\leq|\t-\t'|^{\kappa p}\frac C{t_n^{p/2+1}}\int_0^{t_n}\E\left(K^2(X_s)+K^2(X_{ {\phi(s)}})\right)^{ p/2 }\sigma(X_{s})^{p}ds\leq C|\t-\t'|^{\kappa p}.
\end{align*}

Choosing $p>\frac d{\kappa}$ and  using  the Theorem 20 in the Appendix of \cite{Ibra-Khas} we obtain
\[\frac{1}{n\Delta_n}\sup_{\t\in\Theta}|A_{n,1}(\t)|\stackrel{P}{\longrightarrow} 0\]
 and the statement (i) follows.

\end{proof}
%%%%%%%%%%%%%%%%%%%%%%%%%%%%%%%%%%%%%%%%%%%%%%%%%%%%%%%%%%%%%%%%%%%

%%%%%%%%%%%%%%%%%%%%%%%%%%%
%%%%%%%%%%%%%%%%%%%%%%%%%%%%%%%%%%%%%%%%%%%%%%%%%%%%%%
\section{Auxiliary results}\label{sec:auxlemmas}
%%%%%%%%%%%%%%%%%%%%%%%%%%%%%%%%%%%%%%%%%%%%%
In this section we gather some auxiliary results that are frequently used in our proofs. Furthermore, we give a proof of the ergodicity results of Lemma \ref{lem:ergodic}. We start by some moment inequalities for jump diffusions and their continuous martingale part.
\begin{lem}\label{lem:jac-prot}
Let $X$ satisfy Assumption \ref{ass:existence}. Then for all $t >s$,
\begin{enumerate}
\item
$\forall p\geq 2,$ 
\[
\E[|X_{t}-X_s|^p]^{1/p}\leq C{|t-s|}^{1/p}.
\]
\item  Let $\F_s=\sigma\{X_u, 0\leq u\leq s\}$. Then for $ p\geq 2,$ $p\in\N$,
\begin{equation*}
\E[|X_t-X_s|^p|\F_s]\leq |t-s|(1+|X_s|^p).
\end{equation*}
\item $\forall p>1,$
\[
\E\left[|X^c_t-X^c_s|^p\right]^{1/p}\leq C|t-s|^{1/2}.
\]

\end{enumerate}

\begin{proof}
The first claim follows easily from the two lemmas  and Theorem 66 on p. 339 in
\cite{Protter}.
The second claim follows from Proposition $3.1$ in
\cite{shimizu2006} and the third from the first two lemmas on p.339 in \cite{Protter}.
\end{proof}
\end{lem}
%Below we give the statement of Lemma
%9 from \citet{Genon1993}.
%\begin{lem}\label{lem:Valentine-Jacod}
%Let $e_i^n,$ $U$ be random variables, with $e_i^n $ being $\F_i$ measurable. Suppose that
%\begin{equation}
%\sum_{i=1}^{n}\left|E\left[{e_{i}^n}|\mathcal{F}_{i-1}\right]\right|\to U\quad \mbox{in probability}, \label{eq:cond_conv1}
%\end{equation}
%and
%\begin{equation}
%\sum_{i=1}^{n}E\left[\left(e_i^n\right)^{2}|\mathcal{F}_{i-1}\right]\to0 \quad \mbox{in probability.} \label{eq:cond_conv2}
%\end{equation}
%Then
%\begin{equation*}
%\sum_{i=1}^{n}e_i\to U\quad\mbox{in probability.}
%\end{equation*}
%\end{lem}
\begin{lem}
	\label{L:law_sum_cond_big}
	Under assumptions \ref{ass:existence} to \ref{ass:jumps}, we have for some $C>0,$
	\begin{equation*}
	P((N_n^i)^c \cap (A_n^i)^c) \le C \frac{\Delta_n^2}{v_n/\gamma_{min}} 
	\int_{|z|\geq 3 v_n} \nu(dz).
	\end{equation*}
\end{lem}
\begin{proof}
	We need to introduce some notations. For $z>0$, we define
	$U_z=\int_{t_{i-1}}^{t_i} \int_{|y|\geq 1/z} \mu(ds,dy)$ the number of jumps of $(X_s),\  s\in (t_{i-1},t_i],$ with a size 
	greater than $1/z$, and we set $U_0=0$. It is clear that $(U_z)_{z\ge 0}$ is a process 
	whose increments are independent and distributed with Poisson laws. Hence, it is a Poisson process, and by a simple computation we can show that it has a 
	%non constant 
	jump intensity equal to
	$(t_i-t_{i-1}) z^{-2} (\nu(z^{-1})+\nu(-z^{-1}))$, where $\nu(z)=\nu(dz)/dz$ exists by
	Assumption \ref{ass:jumps} (iii).    	
	
	We define the filtration generated by the process $(U_z)_{z\ge 0}$, by setting for all $z\ge 0$, 
	${\mathcal G}_z=\sigma\{U_y;\  y \leq z\}$. We note $Z_1^*$
	the first jump time of the process $U$, which is a stopping time.
	By construction, we have that $1/Z_1^*$ is the size of the biggest jumps of the L\'evy process
	$L$ on $(t_{i-1},t_i]$, or with the notations of Lemma \ref{lem:filter_infiniteactivity} that, $1/Z_1^* = | \Delta L_{T_i^*}  |$, where $|\Delta L_{T_i^*} |=\max \left \{|\Delta  L_{s} |;\ s\in  ( t_{i-1};t_i]\right\}$.
	
	Moreover, we can write
	\begin{equation*}
	\sum_{t_{i-1}<s\leq t_i; s\neq T_i^*}|\Delta L_s|
	=\int_{t_{i-1}}^{t_i} \int_{|y| < 1/Z_1^*} |y| \mu(ds,dy)
	= \int_{(Z_1^*,\infty)} \frac{1}{z} d U_z,	
	\end{equation*}
	where we have used that $\Delta L_{T_i^*}$ is the only jump with the maximal size $1/Z_1^*$.
	%To estimate 
	%$P({(N_n^{i})}^c\cap (A_n^{i})^c)$ define for $z\in]0,+\infty[$ the mass of $\mu$ bounded away from $0$ by
	%\[
	%U_z:=\int_{t_{i-1}}^{t_i}\int_{]-\infty,-z]\cap[z,+\infty[}\mu(ds,dy)
	%\]
	%and set ${\mathcal G}_z=\sigma\{U_y;\  y> z\}$ and $Z_1^*:=\inf \{z>0, \quad U_z=0\}$.
	%\[
	%{\mathcal G}_z=\sigma\{A\text{-bounded};\quad A\subset [t_{i-1},t_i[\times[z,+\infty[\cup [t_{i-1},t_i[\times ]-\infty,-z]\}
	%\]
	%Loi du deuxième saut conditionnellement au premier:
	%\[
	%P(\mu([t_{i-1},t_i]\times]Z_1^*,Z_{1-h}^*]\geq 1|{\mathcal F}_{Z_1})=1-e^{-\Delta\int_{Z_1}^{Z_1-h}\nu(ds)}
	%\]
	%
	%\[
	%\sum_{s\neq T_i}|\Delta L_s|=\int_{t_i-1}^{t_i}\int_{]0,Z_1^*[ z\mu (ds,dz)
	%\]
	%Using these definitions together with the property of independence of Poisson measures on disjoint bounded sets we obtain
	Hence, we have
	\begin{align*}
	&  P({(N_n^{i})}^c\cap (A_n^{i})^c) = P\left (|\Delta L_{T_i^*}|>\frac {3 v_n}{\gamma_{min}};\quad  \sum_{t_{i-1}<s\leq t_i; s\neq T_i^*}|\Delta L_s| > \frac{v_n}{\gamma_{max}}\right )\\
	&=P\left( (Z_1^*)^{-1}>\frac {3 v_n}{\gamma_{min}}; \quad
	\int_{(Z_1^*,\infty)} z^{-1} d U_z > \frac{v_n}{\gamma_{max}}\right)\\
	&=\E\left [ 1_{\{(Z_1^*)^{-1}>\frac {3 v_n}{\gamma_{min}}\}} 
	P\left(  
	\int_{(Z_1^*,\infty)} z^{-1} d U_z > \frac{v_n}{\gamma_{max}} \mid \mathcal{G}_{Z^*_1}\right)\right]\\
	&\leq \frac{\gamma_{max}}{v_n}  \E\left[ 1_{\{(Z_1^*)^{-1}>\frac {3 v_n}{\gamma_{min}}\}} 
	E\left(  
	\int_{(Z_1^*,\infty)} z^{-1} d U_z  \mid  \mathcal{G}_{Z^*_1}\right)\right],
	\end{align*}
	where we have used the Markov inequality in the last line.
	Using now that $(U_z)_{z \ge 0}$ is a Poisson process with an explicit jump intensity
	$\overline{U}(z):= (t_{i} - t_{i-1}) z^{-2} (\nu(z^{-1}) + \nu(-z^{-1}))$, we deduce,
	\begin{equation*}
	P({(N_n^{i})}^c\cap (A_n^{i})^c)
	\le \frac{\gamma_{max}}{v_n} 
	\E\left[ 1_{\{(Z_1^*)^{-1}>\frac {3 v_n}{\gamma_{min}}\}} 
	E\left(  
	\int_{(Z_1^*,\infty)} z^{-1} \overline{U}(z) dz  \mid  \mathcal{G}_{Z^*_1}\right)\right].
	\end{equation*}
	But, by a simple change of variable, $ \int_{(Z_1^*,\infty)} z^{-1} \overline{U}(z) dz = (t_{i} - t_{i-1}) \int_{|y| < 1/ Z_1^*} |y| \nu (y)  dy \leq \Delta_n \int_{\mathbb{R}} |y| \nu(y) \dd y$. 
	We conclude
	\begin{align*}
	P({(N_n^{i})}^c\cap (A_n^{i})^c)
	&\le \frac{\gamma_{max}}{v_n} \Delta_n \left( \int_{\mathbb{R}} |y| \nu(y) \dd y \right)
	P \left[ (Z_1^*)^{-1}>\frac {3 v_n}{\gamma_{min}} \right]
	\\
	&\le C \frac{\Delta_n}{v_n} P\left(\mu((t_{i-1},t_i] \times [ (-\infty,-\frac{3 v_n}{\gamma_{min}}) \cup (\frac{3 v_n}{\gamma_{min}},+\infty) ] ) \geq 1\right)\\
	&\leq C \frac{\Delta_n^2}{v_n}\int_{|z|>\frac{3v_n}{\gamma_{min}}}\nu(dz),
	\end{align*}
	where $C>0.$ The lemma is proved.
\end{proof}

\begin{prop}\label{prop:ident}
	Under Assumptions \ref{ass:existence} to \ref{ass:jumps}, the Assumption \ref{ass:ident} is equivalent to the condition
	\begin{equation*}
	\forall (\t,\t')\in\Theta^2,\quad \text{such that }\quad \t\neq \t', \quad b(\t,.)\neq b(\t',.).
	\end{equation*}
\end{prop}
\begin{proof}
	It is sufficient to show that if $\mathcal{O}$ is some non empty, open set, then 
	$\pi^\theta(\mathcal{O})>0$.   It is proved in \cite{Mas} (see equation (13) p.43) that for all $\Delta>0$, $x\in \mathbb{R}$, and $\mathcal{O}$ non empty, open set,
	%\begin{equation*}
	$P (X^\theta_\Delta \in \mathcal{O}\mid X^\theta_0=x )>0$.
	%\end{equation*}
	From this, we deduce that  
	\begin{equation*}
	\pi^\theta(\mathcal{O})=\int_{\mathbb{R}} P (X^\theta_\Delta \in \mathcal{O}\mid X^\theta_0=x ) d \pi^\theta(x) >0.
	\end{equation*}
\end{proof}

We conclude this section with a proof of the ergodicity results and moment bounds of Lemma \ref{lem:ergodic}. The proof is based on \cite{Mas}.

\begin{proof}[Proof of Lemma \ref{lem:ergodic}]
Let $q>2,$ $q$ even  and  
%$f^{\star}\geq 0,$ such that 
$f^{\star}(x)=|x|^q.$ 
%if $|x|\geq K,$ $f^{\star}\in\C^{\infty}(\R)$.
%\[ f(x)=\left\{ \begin{array}{ll}
%                     |x|^q\ &si\ |x|>K \\
%                    \mbox {bounded}\ &sinon \\
%                   \end{array}\right.\]
We show that
$f^{\star}$ satisfies the drift condition 
\[\A f^{\star}\leq -c_1f^{\star}+c_2,\]
where $c_1>0,c_2>0.$
%Let $\Q^{\star}$ defined by $(5)$  in \cite {Mas}.
%Obviously $f^{\star}\in\Q^{\star},$ (see the definition of the set $\Q^{\star}$ in $(5)$  of \cite {Mas}). Since for any
%Obviously $f^{\star}\in \Q^{\star}$, so  $\A f^{\star}$ is locally bounded and we only have to show the drift condition for $|x|\geq R>>1$.
Denote
\[\G f(x)=\frac 12 \sigma^2(x)f''(x)+b(\t,x)f'(x),\]
%\[\J_{\star}f(x)=\int_{|z|\leq 1}(f(x+z\gamma(x))-f(x)- f'(x) z\gamma(x))\nu(dz),\]
\[\J f(x)=\int_{\R}(f(x+z\gamma(x))-f(x))\nu (dz).\]
 for any $f$ such that the two previous expressions are defined and decompose 

\[
\A=\G+\J.
\]

Using Taylor's formula together with  Assumptions \ref{ass:ergodic} (iii) and \ref{ass:jumps} (ii) we can write  %\[|\J_{\star}f^{\star}(x)|\leq \frac 12\int_{|z|\leq 1} z^2\gamma^2(x)\sup_{u\in[x,x+z\gamma(x)]}f^{{\star}''}(u)\nu(dz)\leq \frac 12\gamma^2(x)(C+q(q-1)|x|^{q-2})=o(|x|^q)\]
%and
\[|\J f^{\star}(x)|\leq \int_{\R}|z\gamma(x)|\sup_{u\in[x,x+z\gamma(x)]}|f^{{\star}'}(u)|\nu(dz)\leq C\gamma(x)|x|^{q-1}\int_{\R}|z|(1+|z|)^{q-1}\nu(dz) =o(|x|^q)\]
  as $x\to\infty.$
Using Assumption \ref{ass:ergodic} (ii) and (iv)  we get 
\[\G f^{\star}(x)=\frac 12\sigma^2(x)q(q-1)x^{q-2}+b(\t,x)x qx^{q-2}\leq -C|x|^2qx^{q-2}+o(|x|^q)\leq -C qf^{\star}(x)+o(|x|^q),\]
for some $C>0.$
%and hence for all $x\in\R$ 
%\[\A f^{\star} (x)\leq -Cqf^{\star}(x) +o(|x|^q)\]
As $\A f^{\star}(x)$ is locally bounded, using two previous displays we can choose $c_2>0$ and $c_1>0$ such that for all  $x\in\R,$
\[\A f^{\star}(x)\leq -c_1 f^{\star}(x)+c_2.\]
Hence, Assumption $3^{\star}$ from \cite{Mas} holds and  using Theorem $2.2$ from \cite{Mas} we get then 
\begin{equation}\label{eq:born}
\sup_{s\geq 0} \EE[|X_s^{\t}|^q]<\infty
\end{equation}
and using Fatou's lemma results in
\[\sup_{s\geq 0} \EE[|X_{s_-}^{\t}|^q]<\infty.\]
Hence we proved the assertion (3).  
%Now we continue with the proof of (4). 
Using Assumption \ref{ass:recurrence} and the Theorem $2.1$ from \cite{Mas} we get for all $\t\in\Theta$ that $X^{\t}$ admits the unique invariant distribution $\pi^{\theta}$, $f^{\star}\in\L^1(\pi^{\theta})$ and
the ergodic theorem holds. We proved (1) and (2). We continue with the proof of (4).
Using ergodic theorem, for all $q>0,$ 
\[\lim_{t\to\infty}\frac 1t\int_0^{\infty}|X_s^{\t}|^qds=\pi^{\t}(|x|^q), \ \  P-a.s.\]
Moreover, using Jensen's inequality and the bound \eqref{eq:born} we
  get the uniform integrability of the family $\{\frac 1t\int_0^t|X_s^{\t}|^qds,\ t>0\}$:
  \[E\left (\frac 1t\int_0^t|X_s^{\t}|^qds\right )^{1+\eps}\leq \frac 1t\int_0^t[E|X_s^{\t}|^{q(1+\eps)}]ds\leq C,\]
where $C>0,$ and hence 
\[\lim_{t\to\infty}\frac 1t\int_0^{t}E|X_s^{\t}|^{q}ds=\pi^{\t}(|x|^q).\]
\end{proof}

\begin{proof}[Proof of Lemma \ref {lem:Riemann_app}]
Let us first prove (i).  
Using Lemma \ref{lem:jac-prot} $(1)$, with some $\tilde x$ between $X_{t_{i-1}}$ and $X_s$ in the third line below  we obtain:
\begin{align*}
&E\sup_{\theta\in\Theta}\left|\int_{0}^{t_{n}}f(\theta,X_{s})\; ds-\sum_{i=1}^{n}f(\theta,X_{t_{i-1}})\Delta_{i}Id\right| =E\sup_{\theta\in\Theta}\left|\sum_{i=1}^{n}\int_{t_{i-1}}^{t_{i}}f(\theta,X_{s})-f(\theta,X_{t_{i-1}})\; ds\right|\\
&  \leq\sum_{i=1}^{n}\int_{t_{i-1}}^{t_{i}}E\left[\sup_{\theta\in\Theta}\left|f(\theta,X_{s})-f(\theta,X_{t_{i-1}})\right|\right]ds
  \leq \sum_{i=1}^{n}\int_{t_{i-1}}^{t_{i}}E\left[\sup_{\theta\in\Theta}\left |f'(\t,\tilde x )\right|\left|X_{s}-X_{t_{i-1}}\right|\right]\; ds \\
  &\leq \sum_{i=1}^{n}\int_{t_{i-1}}^{t_{i}}\left(E\sup_{\theta\in\Theta}|f'(\t, \tilde x )|^2\right)^{1/2}\left(E |X_{s}-X_{t_{i-1}}|^2\right)^{1/2}\; ds
  \leq C n\Delta_{n}^{3/2}.
\end{align*}

We now prove (ii). We find that
\begin{equation*}\int_{0}^{t_{n}}f(\theta,X_{s})\; ds-\sum_{i=1}^{n}f(\theta,X_{t_{i-1}})\Delta_{i}^nId=\sum_{i=1}^n\int_{t_{i-1}}^{t_i}\left (f(\theta,X_{s})-f(\theta,X_{t_{i-1}})\right )\;ds ,
\end{equation*}
and it is then apparent that this term can be treated exactly as the term $A_{n,3}(\theta)$ given by the equation \eqref{al:decomp_terme3}.
Hence, from \eqref{E:cv_term_A3} (which requires the condition $n\Delta_n^{3-\eps}\to 0$) we have the result.
\end{proof}

\bibliographystyle{plainnat}
\bibliography{bibliography}

\end{document}